\documentclass[11pt]{article}
\usepackage{setspace}
\usepackage{mathrsfs}
\usepackage{mathrsfs}
\usepackage{mathrsfs}
\usepackage{graphicx}
\usepackage{epstopdf}
\usepackage{multirow}
\usepackage{subfigure}

\usepackage{cite}
\usepackage {enumerate}
\usepackage{amsfonts}
\usepackage{amssymb}
\usepackage{amsmath}
\usepackage{amsthm}
\usepackage{algorithm}
\usepackage{algorithmic}

\usepackage[colorlinks=true]{hyperref}
\usepackage{hyperref}
\hypersetup{linkcolor=blue}
\theoremstyle{plain}
\newtheorem{thm}{Theorem}[section]

\newtheorem{lem}[thm]{Lemma}

\makeatletter
\def\@cite#1#2{\textsuperscript{[{#1\if@tempswa , #2\fi}]}}
\makeatother

\theoremstyle{assumption}

\theoremstyle{definition}
\newtheorem{defn}{Definition}[section]

\theoremstyle{remark}
\newtheorem{remk}{Remark}[section]

\begin{document}
	\begin{spacing}{0.5}
	\end{spacing}
	\title{{\Large\bf {An inexact Bregman proximal point method and its acceleration version for unbalanced optimal transport}}
		}
	\author{{\normalsize Xiang Chen     \,\,\,\,Faqiang Wang  \,\,\,\,Jun Liu \,\,\,\,  Li Cui{\thanks{Corresponding author: Li Cui}}}\\}
	\date{{\footnotesize Laboratory of Mathematics and Complex Systems (Ministry of Education), School of Mathematical Sciences, Beijing Normal University, Beijing 100875, People's Republic of China}\\}
	\maketitle
	\begin{minipage}{14cm} {\bf Abstract:}
		{\footnotesize The Unbalanced Optimal Transport (UOT) problem plays increasingly important roles in computational biology, computational imaging and deep learning. Scaling algorithm is widely used to solve UOT due to its convenience and good convergence properties. However, this algorithm has lower accuracy for large regularization parameters, and due to stability issues, small regularization parameters can easily lead to numerical overflow. We address this challenge by developing an inexact Bregman proximal point method for solving UOT. This algorithm approximates the proximal operator using the Scaling algorithm at each iteration. 
			The algorithm (1) converges to the true solution of UOT, (2) has theoretical guarantees and robust regularization parameter selection, (3) mitigates numerical stability issues, and (4) can achieve comparable computational complexity to the Scaling algorithm in specific practice. Building upon this, we develop an accelerated version of inexact Bregman proximal point method for solving UOT by using acceleration techniques of Bregman proximal point method and provide theoretical guarantees and experimental validation of convergence and acceleration. 
		}
	\end{minipage}

	\maketitle
	\numberwithin{equation}{section}
	\newtheorem{theorem}{Theorem}[section]
	\newtheorem{lemma}[theorem]{Lemma}
	\newtheorem{proposition}[theorem]{Proposition}
	\newtheorem{corollary}[theorem]{Corollary}
	\newtheorem{assumption}[theorem]{Assumption}

\section{Introduction}

The Optimal Transport (OT) problem is a problem of finding the optimal cost for transporting mass from one distribution to another \cite{villani2021topics}.  This problem in mathematics and operations research has been proposed for many years. Due to the pioneering work of Brenier\cite{brenier1991polar}, the problem has received renewed attention and has become increasingly popular in the past decade in application areas such as image retrieval and color transfer in computer vision \cite{rabin2015convex,rubner2000earth}, as well as statistical inference in machine learning \cite{solomon2014wasserstein}.

However, the classic OT imposes a very strong constraint, which requires us to normalize the input measure to unit mass. This is an unacceptable constraint for many problems that allow only partial mass transfer, or for handling transfer problems between arbitrary positive measures. In order to address such transfer problems that do not guarantee mass conservation, a modified version of the OT problem was proposed, known as the Unbalanced Optimal Transport (UOT) problem. This problem is obtained by relaxing the hard constraint of optimal transport to a soft constraint \cite{frogner2015learning}.
In recent years, the problem of UOT has been widely applied in computational biology \cite{schiebinger2019optimal}, computer imaging \cite{lee2019parallel}, full waveform inversion \cite{li2022application}, deep learning \cite{yang2018scalable}, and statistics \cite{janati2020spatio}.

In this paper, we focus on the computation of the discrete UOT problem using KL divergence to relax constraints.
\begin{equation}
	\min_{\bf{P}\geqslant 0}\langle\bf {C}, \bf{P}\rangle+
	\lambda_{1} \mathrm{KL}(\bf{P}\bf{1}_{m}\mid \bf{a})+
	\lambda_{2} \mathrm{KL}(\bf{P}^{T}\bf{1}_{n} \mid \bf{b}).
	\label{eq:uot}
\end{equation}
Where $\bf{a}$,$\bf{b}$ are two positive vectors. Matrix $\bf{C}=[ c_{ij} ]\in\mathbb{R}_{+}^{n \times m}$ is the cost matrix, whose element $c_{ij} $ represents the distance between the $i$-th support point of
$\bf{a}$ and the $j$-th support point of $\bf{b}$. $\lambda_{1}$,$\lambda_{2}$ are unbalancedness parameters. KL$(\cdot \mid \cdot)$ is defined by KL$(\bf{x} \mid \bf{y})=\sum_{i} \bf{x}_{i}\log \frac{\bf{x}_{i}}{\bf{y}_{i}} -\bf{x}_{i} +\bf{y}_{i}$. Notation $\langle \cdot ,\cdot \rangle$ represents the Frobenius dot-product. Finally, $\bf{1}_{n} $ represents $n$-dimensional vector of ones.

In order to facilitate the solution of OT problems, Cuturi proposed to use the entropy of the transport plan to regularize the objective function, based on which a fast algorithm for solving OT problems, i.e., the Sinkhorn algorithm \cite{cuturi2013sinkhorn}. Similar to the process of extending the OT problem to the UOT problem, this method has also been extended to solve UOT problems \cite{chizat2018scaling}. That is, by adding an entropy term to the UOT problem to regularize the original problem. For example, for the UOT problem defined by KL divergence, its entropy regularized problem is represented as follows.

\begin{equation}
		\min_{\bf{P}\geqslant 0}\langle\bf {C}, \bf{P}\rangle+
	\lambda_{1} \mathrm{KL}(\bf{P}\bf{1}_{m}\mid \bf{a})+
	\lambda_{2} \mathrm{KL}(\bf{P}^{T}\bf{1}_{n} \mid \bf{b})+
	\epsilon h(\bf{P}),
	\label{eq:euot}
\end{equation}
where $h(\bf{P})=\sum_{i,j}\bf{P}_{i,j}(\log \bf{P}_{i,j}-1)$ is the entropic regularizer. Then the problem can be solved rapidly by the Scaling algorithm,

\begin{equation}
	\bf{u}^{(l+1)}=(\frac{\bf{a}}{\bf{K} \bf{v}^{(l)}})^{\frac{\lambda_{1}}{\lambda_{1}+\epsilon}}, \quad \bf{v}^{(l+1)}=(\frac{\bf{b}}{\bf{K}^{T} \bf{u}^{(l+1)}})^{\frac{\lambda_{2}}{\lambda_{2}+\epsilon}},
	\label{eq:scalg}
\end{equation}
starting from $\bf{v}^{(l)}=\bf{1}_{m}$, where $\bf{K}=[K_{i,j}]$ and $ \bf{K}_{i,j}=e^{-\bf{C}_{i j} / \epsilon}$. And the optimal solution $\bf{P}^{*}$ has the form $\bf{P}^{*}_{i,j}=\bf{u}_{i}\bf{K}_{i,j}\bf{v}_{j} $. This algorithm is actually based on the well-known Dykstra's algorithm \cite{bauschke2000dykstras,benamou2015iterative}. In fact, the Scaling algorithm is proven to achieve a $\mathcal{O}(n^2)$ complexity \cite{altschuler2017near,pham2020unbalanced}.

Although the Scaling algorithm can be solved through matrix vector products, with a simple and easily implementable form and good computational complexity, it also has several drawbacks. Firstly, for the entropy regularization parameter $\epsilon$ in \eqref{eq:euot}, when it is set to a larger value, the algorithm may converge quickly, but the error of the approximate solution obtained will be large. Secondly, if we desire an approximate solution with a small error, we need to choose a very small $\epsilon$, but due to $\bf{K}_{i,j}=e^{-\bf{C}_{i j} / \epsilon} $, a small $\epsilon$ may lead to numerical overflow when computing $\bf{K}$, although this issue can be resolved by calculating in log-space , it requires additional exponential and logarithmic operations, sacrificing the efficiency advantage of the algorithm \cite{chizat2018scaling,benamou2015iterative}.
In addition, since the contraction ratio of the Scaling algorithm is determined by $(1+\frac{\epsilon}{\lambda})^{-1}$, a small $\epsilon$ will cause a sharp increase in the number of iterations of the algorithm. Finally, the Scaling algorithm will also sacrifice the sparsity of the solution, which is disadvantageous for the application of UOT.

So, can we use a moderate entropy regularization parameter $\epsilon$ to obtain a relatively accurate approximate solution without increasing the computational cost? In the work of Xie and YANG \cite{xie2020fast,yang2022bregman}, they solve the original OT problem by using the Bregman proximal point method, where the subproblems constructed in the iterative step can be formulated as an entropy-regularized optimal transport problem, and the approximate solution to the subproblems can be obtained using the Sinkhorn algorithm.
The IPOT algorithm and the iEPPA algorithm constructed in this way can both use a moderately sized entropy regularization parameter $\epsilon$ to obtain a relatively accurate approximate solution to the OT problem, and have good sparsity. This give us a lot of inspiration.  We have applied this idea to UOT problem and obtained a new algorithm, i.e, Inexact Bregman Proximal point method for solving Unbalanced Optimal Transport (IBPUOT). 

Since IBPUOT is a first order algorithm, we can use he acceleration technique to accelerate the algorithm IBPUOT. Accelerated thinking originated from Nesterov, whose work inspired various extensions and accelerated variants \cite{nesterov1988approach,nesterov1983method}. For example, the classic accelerated proximal point method \cite{guler1992new}, as well as the recent accelerated Bregman proximal point method \cite{yang2022bregman,yan2020bregman}, and accelerated variants of Bregman proximal gradient method \cite{hanzely2021accelerated}. Based on these works, we naturally developed an Accelerated version of IBPUOT called AIBPUOT. In fact, although IBPUOT can accelerate convergence speed by taking small proximal parameters, small proximal parameters will lead to increased internal iterations in solving UOT problems, thus preventing acceleration.Therefore, it is necessary for us to develop an acceleration variant AIBPUOT that does not require the use of smaller proximal parameters.

The contributions of this paper are summarized as follows.
\begin{itemize}
	\item In this paper, we propose a new method for solving the UOT problem defined by KL divergence based on the generalized proximal point method using Bregman divergence --- IBPUOT. 
	This algorithm can select a moderately sized $\epsilon$ to obtain the exact solution of UOT. The convergence proof of the algorithm IBPUOT is provided, and it is proven that the convergence rate of IBPUOT is $\mathcal{O}(\frac{1}{N})$. 
	
	\item We also combine Nesterov's acceleration technique to give the specific form of the accelerated version of IBPUOT, we call it AIBPUOT. And  we complete the corresponding convergence proof and convergence rate analysis. By making use of the triangle scaling property of the Bregman distance, we prove that the convergence rate of AIBPUOT is $\mathcal{O}(\frac{1}{N^{\gamma}})$, where $\gamma$ is the triangle scaling exponent.
	
	\item We conducted numerical experiments to compare the performance of IBPUOT and the Scaling algorithm under different entropy regularization parameters, as well as the acceleration effect of AIBPUOT.

\end{itemize}

Next, we will provide notation and preliminaries in the Section \ref{sec:2}. We present the specific form of IBPUOT and its convergence proof, and prove that its convergence rate in Section \ref{sec:3} . We give the form of AIBPUOT and complete the corresponding convergence proof and convergence rate analysis in Section \ref{sec:4}. We presented some numerical results in Section \ref{sec:5}.

\section{Notation and preliminaries}
\label{sec:2}
For a vector $\bf{x} \in \mathbb{R}^{n} $, we denote its $i$-th entry as $\bf{x}_{i}$ , and the diagonal matrix is denoted as $Diag(\bf{x})$, whose $i$-th diagonal entry is $\bf{x}_{i}$. 
Denote $\odot$ as element-wise matrix multiplication, $\frac{(\cdotp)}{(\cdotp)}$ as element-wise division. 
For a given function $f:\mathbb{R}^{n}\rightarrow\mathbb{R}$, dom $f$ :=$\{\mathbf{x} \in \mathbb{R}^{n},f(x)<\infty\}$

\begin{defn}[$\nu $ -subdifferential\cite{chu2023efficient}]
	For a proper convex function $f:\mathcal{X}\subseteq\mathbb{R}^{n}\rightarrow\mathbb{R}$, give a $\nu \geq 0 $, the $\nu $-subdifferential of $f$ at $\bf{x} \in$ dom $f$ is defined by
	\begin{equation*}
	\partial_{\nu}f:=\{ \mathbf{d} \in \mathbb{E}: \mathit{f}(\bf{y})\geqslant \mathit{f}(\bf{x})+\langle\bf{d},\bf{y}-\bf{x}\rangle-\nu,\forall \bf{y} \in \mathcal{X}\},
	\end{equation*}
	when $\nu=0$, $ \partial_{\nu}f$ is denoted by $\partial f$
\end{defn}

\begin{defn}[Legendre function\cite{xie2020fast}]
	Let $h:\mathcal{X}\subseteq\mathbb{R}^{n}\rightarrow\mathbb{R}$ be a lsc proper convex function. When it satisfies the following properties, we call it Legendre function
	\begin{itemize}
		\item [1)]Essentially smooth: if $h$ is differentiable on int dom $h$, for every sequence $\{\mathbf{x}^{k}\}\subseteq$ int dom $h$ converging to a boundary point of dom $h$ as $k\rightarrow +\infty$, $\|\nabla h(\mathbf{x}^{k})\|\rightarrow \infty$
		\item [2)]Legendre type: if $h$ is essentially smooth and strictly convex on int dom $h$.
	\end{itemize}
\end{defn}

\begin{defn}[Bregman distance\cite{xie2020fast}]
	For any given Legendre function $h:\mathcal{X}\subseteq\mathbb{R}^{n}\rightarrow\mathbb{R}$, and for any $ \mathbf{x} \in$ dom $h$, $\mathbf{y}\in$ int dom $h$,
	\begin{equation}
			D_{h}(\mathbf{y},\mathbf{x})=h(\mathbf{y})-h(\mathbf{x})-\langle\nabla h(\mathbf{x}),\mathbf{y}-\mathbf{x}\rangle.
	\end{equation}
	It is easy to see that $D_{h} (\mathbf{x} , \mathbf{y}) \geq 0$ and equality holds if and only if $\mathbf{x}=\mathbf{y}$ due to the strictly convexity of $h$.
\end{defn}

\begin{lem}[three points identity\cite{xie2020fast}]
	Given a proper closed strictly convex function $h:\mathcal{X}\subseteq\mathbb{R}^{n}\rightarrow\mathbb{R}$, $D_{h}$ is a general Bregman distance, and $\mathbf{x},\mathbf{y},\mathbf{z}\in \mathcal{X}$ such that $h(\mathbf{x}),h(\mathbf{y}),h(\mathbf{z})$ are finite and $h$ is differentiable at $\mathbf{y}$ and $\mathbf{z}$,
	\begin{equation}
			D_{h}(\mathbf{x},\mathbf{z})-D_{h}(\mathbf{x},\mathbf{y})-D_{h}(\mathbf{y},\mathbf{z})=\langle\nabla h(\mathbf{y})-\nabla h(\mathbf{z}),\mathbf{x}-\mathbf{y}\rangle.
			\label{three point}
	\end{equation}
\end{lem}

\begin{lem}[see\cite{lemaire1995convergence}]
	Suppose that $\{\mu_{n}\}_{n=0}^{\infty} \subseteq \mathbb{R}_{+}$  and $\{\beta_{n}\}_{n=0}^{\infty} \subseteq \mathbb{R}$ are two sequences. Let $\tau_{n}:=\sum_{k=0}^{n}\mu_{k}$ and $\alpha_{n}:=\tau_{n}^{-1}\sum_{k=0}^{n}\mu_{k}\beta_{k}$.
	\item (1) $\liminf _{n \rightarrow+\infty} \beta_{n} \leqslant \liminf _{n \rightarrow+\infty} \alpha_{n} \leqslant \limsup _{n \rightarrow+\infty} \alpha_{n} \leqslant \limsup _{n \rightarrow+\infty} \beta_{n}$.
	\item (2) If $\lambda_{n}\rightarrow+\infty$, $\beta:=\lim_{n\rightarrow+\infty}\beta_{n}$ exists, then $\alpha_{n}\rightarrow \beta$.
	\label{lem:sequence}
\end{lem}

Consider the following convex optimization problem,
\begin{equation}
		\min_{\mathbf{x}}f(\mathbf{x})\quad\text{s.t.}\quad \mathbf{x}\in  \overline{\mathcal{X}},
\end{equation}
where $f:\overline{\mathcal{X}}\subseteq\mathbb{R}^{n}\rightarrow \mathbb{R}$ is a proper closed convex function, $\mathcal{X}$ is a nonempty convex open set, $\overline{\mathcal{X}}$ is the closure of $\mathcal{X}$.

There are currently many methods to solve this convex optimization problem, among which the simplest and most basic method is the proximal point method. Now we introduce its general form, the objective of the generalized proximal point method is to solve
\begin{equation}
	\arg \min _{\mathbf{x}\in \overline{\mathcal{X}}} f(\mathbf{x}).
	\label{ppap}
\end{equation}

The algorithm generates a sequence $\{\mathbf{x}^{k}\}$ by the following generalized
proximal point iterations to solve Problem \eqref{ppap},
\begin{equation}
		\bf{x}^{k+1}= \arg \min_{\bf{x} \in \overline{\mathcal{X}}} \mathit{f}(\bf{x})+\beta_{k} \mathit{d}(\bf{x},\bf{x}^{k}),
	\label{eq:ppa}
\end{equation}
where $\beta_{t}>0$ is a given proximal parameter, and $d$ is a regularization term used to define the proximal operator. When $d(\bf{x},\bf{y})=\frac{1}{2}\| \bf{x}-\bf{y} \|_{2}^{2}$, it is the classical proximal point method.

If we choose $d(\bf{x},\bf{y})=\mathit{D}_{\mathit{h}}(\bf{x},\bf{y})$, equation \eqref{eq:ppa} becomes the following form
\begin{equation}
	\bf{x}^{k+1}= \arg \min_{\bf{x} \in \overline{\mathcal{X}}} \mathit{f}(\bf{x})+\beta_{k} \mathit{D}_{\mathit{h}}(\bf{x},\mathbf{x}^{k}).
	\label{eq:bppa}
\end{equation}

This method is commonly referred to as the Bregman proximal point algorithm (BPPA). As the iterate $\bf{x}^{k+1}$ is obtained by approximately solving a subproblem during the iterative process, we need to establish a termination condition to end the process of approximating the subproblem. In this regard, a framework based on $\nu$-th differentiability is widely used \cite{kiwiel1997proximal,teboulle1997convergence}, and can be applied to the inexact Bregman proximal point algorithm as follows.
\begin{equation}
	0 \in \partial_{\nu_{k}}f(\bf{x}^{k+1})+\beta_{k}(\nabla \mathit{h}(\mathbf{x}^{k+1})- \nabla \mathit{h}(\mathbf{x}^{k})).
\end{equation}

The proximal point method has a robust convergence behavior for the parameter $\beta$. As for $\beta_{k}$, its specific choice generally only affects the convergence speed of the algorithm, and does not affect the convergence of the algorithm. In other words, there can be a relatively broad condition for the choice of $\beta_{k}$, which still ensures the accuracy of the algorithm. Even if the proximal operator defined in \eqref{eq:ppa} cannot be precisely solved during iteration, the algorithm can still guarantee global convergence under certain conditions \cite{solodov2001unified,schmidt2011convergence}. This is also the key reason why we choose to use the inexact proximal point method to solve the problem of UOT.

\section{An inexact Bregman proximal point algorithm for unbalanced optimal transport}
\label{sec:3}
In this section, we explain how to use the Bregman proximal point method to solve the UOT problem \eqref{eq:uot}, develop a new algorithm IBPUOT, and provide a convergence proof and convergence rate analysis of IBPUOT. In section \ref{sec:3.1}, we explained the design concept and specific form of IBPUOT. In section \ref{sec:3.2}, we provide the convergence analysis of IBPUOT. 

\subsection{IBPUOT}
\label{sec:3.1}
Recall the Bregman proximal point iteration \eqref{eq:bppa}, 
since problem \eqref{eq:uot} is convex and $\{\mathbf{P}:\mathbf{P}\geqslant0\}$ is a closed convex set, we take $f(\mathbf{P})=\langle\bf {C}, \bf{P}\rangle+
\lambda_{1} \mathrm{KL}(\bf{P}\bf{1}_{m}\mid \bf{a})+
\lambda_{2} \mathrm{KL}(\bf{P}^{T}\bf{1}_{n} \mid \bf{b})$, $\overline{\mathcal{X}}=\{\mathbf{P}:\mathbf{P}\geqslant0\}$, and d is the Bregman divergence $D_{h}$ with respect to the entropy function $h(\mathbf{x})=\mathbf{x}_{i}(\log\mathbf{x}_{i}-1)$. 
Therefore, $D_{h}$ has the following form,
\begin{equation}
	D_{h}(\mathbf{x}, \mathbf{y})=\sum_{i} \mathbf{x}_{i} \log \frac{\mathbf{x}_{i}}{\mathbf{y}_{i}}-\sum_{i} \mathbf{x}_{i}+\sum_{i} \mathbf{y}_{i}.
	\label{eq:bregd}
\end{equation}

As a result, problem \eqref{eq:uot} is solved by the following Bregman proximal point iterations,
\begin{equation}
	\mathbf{P}^{k+1} = \arg \min_{\mathbf{P}\geqslant 0}
	\langle\bf {C}, \bf{P}\rangle+
	\lambda_{1} \mathrm{KL}(\bf{P}\bf{1}_{m}\mid \bf{a})+
	\lambda_{2} \mathrm{KL}(\bf{P}^{T}\bf{1}_{n} \mid \bf{b})+
	\beta_{k}\mathit{D}_{\mathit{h}}({\mathbf{P},\mathbf{P}^{k}}).
	\label{breguot}
\end{equation}

Substitute Bregman divergence \eqref{eq:bregd} into iteration \eqref{breguot}, then iteration \eqref{breguot} can be rewritten as
\begin{equation}
	\mathbf{P}^{k+1}=\arg \min_{\mathbf{P}\geqslant 0}\langle\bf{C} -\beta_{k}\log\mathbf{P}^{k}, \bf{P}\rangle+
	\lambda_{1} \mathrm{KL}(\bf{P}\bf{1}_{m}\mid \bf{a})+
	\lambda_{2} \mathrm{KL}(\mathbf{P}^{T}\bf{1}_{n} \mid \bf{b})+
	\beta_{k} \mathit{h}(\bf{P}).
	\label{eq:beuot1}
\end{equation}

Let $\mathbf{C}^{k}=\mathbf{C} -\beta_{k}\log\mathbf{P}^{k}$, we can get
\begin{equation}
	\mathbf{P}^{k+1}=\arg \min_{\mathbf{P}\geqslant 0}\langle\mathbf{C}^{k}, \bf{P}\rangle+
	\lambda_{1} \mathrm{KL}(\bf{P}\bf{1}_{m}\mid \bf{a})+
	\lambda_{2} \mathrm{KL}(\bf{P}^{T}\bf{1}_{n} \mid \bf{b})+
	\beta_{k} \mathit{h}(\bf{P}).
	\label{eq:beuot2}
\end{equation}

As we can see, problem \eqref{eq:beuot2} is a entropy regularized UOT problem, Similar to \eqref{eq:euot}, problem \eqref{eq:beuot2} can also be approximated using the scaling algorithm. we should replace $\mathbf{K}_{i,j}$ by $\mathbf{K}_{i,j}^{k}=e^{-\mathbf{C}_{i j}^{k} /\beta_{k}} =\mathbf{P}_{i j}^{k} e^{-\mathbf{C}_{i j} / \beta_{k}}$.

In fact, within the framework of the inexact Bregman proximal point method, we only need to use the Scaling algorithm to approximately solve the problem \eqref{eq:beuot2} and generate a sequence $\{\mathbf{P}^{k}\}$ that satisfies the following inexact conditions in order to make the entire algorithm converge.
\begin{equation}
		0 \in \partial_{\nu_{k}}f(\bf{P}^{k+1})+\beta_{k}(\nabla \mathit{h}(\mathbf{P}^{k+1})- \nabla \mathit{h}(\mathbf{P}^{k})).
		\label{eq:terminate}
\end{equation}
The complete algorithm is presented as Algorithm \ref{alg:IBPUOT}, and for simplicity, we can take $\beta_{k}$ as a constant, and let $\beta_{k}=\beta$.

\begin{algorithm}[!h]
	\caption{IBPUOT($\mathbf{a},\mathbf{b},\lambda_{1},\lambda_{2},\mathbf{C}$)}
	\label{alg:IBPUOT}
	\renewcommand{\algorithmicrequire}{\textbf{Input:}}
	\renewcommand{\algorithmicensure}{\textbf{Output:}}
	\begin{algorithmic}
		\REQUIRE positive vectors $\mathbf{a}$, $\mathbf{b}$, unbalancedness parameters $\lambda_{1}$,$\lambda_{2}$ and cost matrix $\mathbf{C}$
		\ENSURE $\mathbf{P}^{k}$
		\STATE $\bf{v}=\bf{1}_{m}$
		\STATE $\mathbf{K}=e^{-\mathbf{C}/\beta}$
		\STATE $\mathbf{P}^{0}=\mathbf{1}_{n}\mathbf{1}_{m}^{T}$
		\FOR{k=0,1,3,\dots,N}
			\STATE $\mathbf{G}=\mathbf{K}\odot \mathbf{P}^{k}$
			\WHILE{termination criterion \eqref{eq:terminate} is not met}
				\STATE $\bf{u}=(\frac{\bf{a}}{\bf{K}\bf{v}})^{\frac{\lambda_{1}}{\lambda_{1}+\beta}}$
				\STATE $\bf{v}=(\frac{\bf{b}}{\bf{K}^{T} \bf{u}})^{\frac{\lambda_{2}}{\lambda_{2}+\beta}}$
			\ENDWHILE
			\STATE $\mathbf{P}^{k+1}=Diag(\bf{u})\mathbf{G} \mathit{Diag}(\bf{v})$
		\ENDFOR
		\RETURN	$\mathbf{P}^{k+1}$
	\end{algorithmic}		
\end{algorithm}

\subsection{Convergence analysis of IBPUOT}
\label{sec:3.2}
Next, we will establish the convergence of Algorithm \ref{alg:IBPUOT} IBPUOT. Our proof is inspired by some related work \cite{xie2020fast,yang2022bregman}. First, for convenience, we will use $f(\mathbf{P})$ to represent problem \eqref{eq:uot} and $h$ to present the entropy function. And then we provide a sufficient descent property through the following lemma.

\begin{lem}[sufficient descent property]
	Let $\{\mathbf{P}^{k}\}$ be the sequence generated by the IBPUOT in Algorithm \ref{alg:IBPUOT}, and for any $\mathbf{P} \in dom f$ and $\beta_{k}>0$, we have:
	\begin{equation}
			f(\mathbf{P}^{k+1}) \leqslant f(\mathbf{P})+\beta_{k}\left(D_{h}(\mathbf{P}, \mathbf{P}^{k})-D_{h}(\mathbf{P}, \mathbf{P}^{k+1})-D_{h}(\mathbf{P}^{k+1}, \mathbf{P}^{k})\right) +\nu_{k}.
			\label{eq:descent}
	\end{equation}
	\label{lem:descent}
\end{lem}
\begin{proof}
	according to condition \eqref{eq:terminate}, we know that there exists a $\mathbf{d}^{k+1}\in\partial_{\nu_{k}}f(\mathbf{P}^{k+1})$ that satisfies $0=\mathbf{d}^{k+1}+\beta_{k}(\nabla \mathit{h}(\mathbf{P}^{k+1})- \nabla \mathit{h}(\mathbf{P}^{k}))$. therefor, for any $\mathbf{P} \in dom f$, from the definition of the $\nu$-subdifferential of $f$, we can know
	\begin{equation*}
		\begin{aligned}
			f(\mathbf{P}) &\geqslant f(\mathbf{P}^{k+1})+\langle\mathbf{d}^{k+1},\mathbf{P}-\mathbf{P}^{k+1} \rangle -\nu_{k} \\
			&=f(\mathbf{P}^{k+1})+\langle -\beta_{k}(\nabla \mathit{h}(\mathbf{P}^{k+1})- \nabla \mathit{h}(\mathbf{P}^{k})),\mathbf{P}-\mathbf{P}^{k+1} \rangle -\nu_{k}.
		\end{aligned}
	\end{equation*}
	So we can obtain
	\begin{equation*}
		f(\mathbf{P}^{k+1}) \leqslant f(\mathbf{P})+\beta_{k}\langle\nabla \mathit{h}(\mathbf{P}^{k+1})- \nabla \mathit{h}(\mathbf{P}^{k}),\mathbf{P}-\mathbf{P}^{k+1} \rangle +\nu_{k}.
	\end{equation*}
	According to the three points identity \eqref{three point}, we can get
	\begin{equation*}
		f(\mathbf{P}^{k+1}) \leqslant f(\mathbf{P})+\beta_{k}\left(D_{h}(\mathbf{P}, \mathbf{P}^{k})-D_{h}(\mathbf{P}, \mathbf{P}^{k+1})-D_{h}(\mathbf{P}^{k+1}, \mathbf{P}^{k})\right) +\nu_{k}.
	\end{equation*}
\end{proof}

On the basis of Lemma \ref{lem:descent}, we can further prove the convergence of the Algorithm \ref{alg:IBPUOT} IBPUOT.

\begin{thm}
Let $\{\mathbf{P}^{k}\}$ be the sequence generated by the IBPUOT in Algorithm \ref{alg:IBPUOT}, for any $\mathbf{P}^{*}$, which is the optimal solution of problem \eqref{eq:uot}, i.e.,$f(\mathbf{P})$, we have
\begin{equation}
	f(\mathbf{P}^{N})-f(\mathbf{P}^{*})\leqslant\tau_{N-1}^{-1}\left( D_{h}(\mathbf{P}^{*},\mathbf{P}^{0})+\sum_{k=0}^{N-1}\beta_{k}^{-1}\nu_{k}+\sum_{k=0}^{N-1}\tau_{k-1}\eta_{k}\right),
	\label{eq:thm3.2}
\end{equation}
where $\tau_{-1}:=0$, $\tau_{k}:=\sum_{t=0}^{k}\beta_{t}^{-1}$, and $\eta_{k}:=f(\mathbf{P}^{k+1})-f(\mathbf{P}^{k})\leqslant \nu_{k}$ for any integer $k \geqslant 0$. In addition, if $\sup_{k}\{\beta_{k}\}<\infty$ and $\sum_{k}\nu_{k}<\infty$, then $f(\mathbf{P}^{k})\rightarrow f^{*}:=f(\mathbf{P}^{*})$.
\end{thm}
\begin{proof}
	For the \eqref{eq:descent} in the lemma \ref{lem:descent}, if we substitute the $\mathbf{P}$ to $\mathbf{P}^{k}$, we have
	\begin{equation*}
	f(\mathbf{P}^{k+1}) \leqslant f(\mathbf{P}^{k})-\beta_{k}\left(D_{h}(\mathbf{P}^{k}, \mathbf{P}^{k+1})+D_{h}(\mathbf{P}^{k+1}, \mathbf{P}^{k})\right) +\nu_{k}.
	\end{equation*}
	Therefore, for any $\eta_{k}$, there exists
	\begin{equation*}
		\begin{aligned}
			\eta_{k}&:=f(\mathbf{P}^{k+1})-f(\mathbf{P}^{k})\\
			&\leqslant -\beta_{k}\left(D_{h}(\mathbf{P}^{k}, \mathbf{P}^{k+1})+D_{h}(\mathbf{P}^{k+1}, \mathbf{P}^{k})\right) +\nu_{k} \\
			&\leqslant \nu_{k}.
		\end{aligned}
	\end{equation*}
	Because of $\tau_{-1}:=0$, $\tau_{k}:=\sum_{t=0}^{k}\beta_{t}^{-1}=\tau_{k-1}+\beta_{k}^{-1}$, for any $k \geqslant 0$, we get
	\begin{equation*}
		(\tau_{k}-\beta_{k}^{-1})f(\mathbf{P}^{k+1})=\tau_{k-1}f(\mathbf{P}^{k})+\tau_{k-1}\eta_{k}.
	\end{equation*}
	 Furthermore, we can obtain 
	\begin{equation*}
		\beta_{k}^{-1}f(\mathbf{P}^{k+1})=\tau_{k}f(\mathbf{P}^{k+1})-\tau_{k-1}f(\mathbf{P}^{k})-\tau_{k-1}\eta_{k}.
	\end{equation*}
	Summing the above expression from $k=0 $ to $k=N-1 $, we have
	\begin{equation}
		\sum_{k=0}^{N-1} \beta_{k}^{-1}f(\mathbf{P}^{k+1})=\tau_{N-1}f(\mathbf{P}^{N})
		-\sum_{k=0}^{N-1}\tau_{k-1}\eta_{k}.
		\label{eq:thm3.2.1}
	\end{equation}
	Assume $\mathbf{P}^{*}$ is an optimal solution of $f(\mathbf{P})$, we let $\mathbf{P}=\mathbf{P}^{*}$ in the \eqref{eq:descent} to obtain
	\begin{equation*}
	\begin{aligned}
				f(\mathbf{P}^{k+1})-f(\mathbf{P}^{*})&\leqslant \beta_{k}\left(D_{h}(\mathbf{P}^{*}, \mathbf{P}^{k})-D_{h}(\mathbf{P}^{*}, \mathbf{P}^{k+1})-D_{h}(\mathbf{P}^{k+1}, \mathbf{P}^{k})\right) +\nu_{k}\\
				&\leqslant\beta_{k}\left(D_{h}(\mathbf{P}^{*}, \mathbf{P}^{k})-D_{h}(\mathbf{P}^{*}, \mathbf{P}^{k+1}) \right)+\nu_{k}.
	\end{aligned}
	\end{equation*}
	So we can get
	\begin{equation*}
		\beta_{k}^{-1}f(\mathbf{P}^{k+1})-\beta_{k}^{-1}f(\mathbf{P}^{k})
		\leqslant D_{h}(\mathbf{P}^{*}, \mathbf{P}^{k})-D_{h}(\mathbf{P}^{*}, \mathbf{P}^{k+1})+\beta_{k}^{-1}\nu_{k}.
	\end{equation*}
		Summing the above expression from $k=0 $ to $k=N-1 $, we have
	\begin{equation}
		\begin{aligned}
			\sum_{k=0}^{N-1}\beta_{k}^{-1}f(\mathbf{P}^{k+1})-\tau_{N-1}f(\mathbf{P}^{*})&\leqslant D_{h}(\mathbf{P}^{*}, \mathbf{P}^{0})-D_{h}(\mathbf{P}^{*}, \mathbf{P}^{N})+\sum_{k=0}^{N-1}\beta_{k}^{-1}\nu_{k}\\
			&\leqslant D_{h}(\mathbf{P}^{*}, \mathbf{P}^{0})+\sum_{k=0}^{N-1}\beta_{k}^{-1}\nu_{k}.
		\end{aligned}
		\label{eq:thm3.2.2}
	\end{equation}
	We can combine \eqref{eq:thm3.2.1} and \eqref{eq:thm3.2.2} to get
	\begin{equation*}
		\tau_{N-1}\left(f(\mathbf{P}^{N})-f(\mathbf{P}^{*}) \right)\leqslant D_{h}(\mathbf{P}^{*}, \mathbf{P}^{0})+\sum_{k=0}^{N-1}\beta_{k}^{-1}\nu_{k}+\sum_{k=0}^{N-1}\tau_{k-1}\eta_{k}.
	\end{equation*}
	And then we Divide the above inequality by $\tau_{N-1}$, we can get
	\begin{equation*}
			f(\mathbf{P}^{N})-f(\mathbf{P}^{*})\leqslant\tau_{N-1}^{-1}\left( D_{h}(\mathbf{P}^{*},\mathbf{P}^{0})+\sum_{k=0}^{N-1}\beta_{k}^{-1}\nu_{k}+\sum_{k=0}^{N-1}\tau_{k-1}\eta_{k}\right).
	\end{equation*}
	
	Now we have completed the proof of \eqref{eq:thm3.2}. Next, we are going to prove that if $\sup_{k}\{\beta_{k}\}<\infty$ and $\sum\nu_{k}<\infty$, then $f(\mathbf{P}^{k})\rightarrow f^{*}:=f(\mathbf{P}^{*})$. 
	
	For any non-negative integer $n$, according to \eqref{eq:thm3.2.2}, we obtain that
	\begin{equation}
		\tau_{n}^{-1}\sum_{k=0}^{n}\beta_{k}^{-1}f(\mathbf{P}^{k+1})\leqslant f(\mathbf{P}^{*})+\tau_{n}^{-1}D_{h}(\mathbf{P}^{*}, \mathbf{P}^{0})+\tau_{n}^{-1}\sum_{k=0}^{n}\beta_{k}^{-1}\nu_{k},
	\end{equation}
Because $\sup_{k}\{\beta_{k}\}<\infty$, $\tau_{n}\rightarrow +\infty$ when $n\rightarrow +\infty$. And $\nu_{k}\rightarrow0$ when $k\rightarrow+\infty$ since $\sum\nu_{k}<\infty$. Thus, according to the Lemma \ref{lem:sequence}, we obtain that $\tau_{n}^{-1}\sum_{k=0}^{n}\beta_{k}^{-1}\nu_{k}\rightarrow0$ when $n\rightarrow +\infty$, and the following inequality
\begin{equation*}
	\liminf_{n \rightarrow+\infty}f(\mathbf{P}^{n+1}) \leqslant \liminf_{n \rightarrow+\infty}\tau_{n}^{-1}\sum_{k=0}^{n}\beta_{k}^{-1}f(\mathbf{P}^{k+1}) \leqslant f(\mathbf{P}^{*}).
\end{equation*}
And because $f(\mathbf{P}^{*}):=\min f(\mathbf{P})\leqslant f(\mathbf{P}^{n+1})$, so we can get that $\liminf_{n \rightarrow+\infty}f(\mathbf{P}^{n+1})=f(\mathbf{P}^{*})$.
Therefore, $f(\mathbf{P}^{k})\rightarrow f^{*}:=f(\mathbf{P}^{*})$
	
\end{proof}

Finally, we discuss the convergence rate of IBPUOT. According to \eqref{eq:thm3.2}, if the summable-error condition are finite,.i.e, $\sum_{k=0}^{N-1}\beta_{k}^{-1}\nu_{k}<\infty,\sum_{k=0}^{N-1}\tau_{k-1}\eta_{k}<\infty$.Then the convergence rate of IBPUOT is controlled by $\tau_{N-1}^{-1}:=(\sum_{t=0}^{N-1}\beta_{t}^{-1})^{-1}$. If we choose $\beta_{k}$ as a constant $\beta$, then the convergence rate of algorithm IBPUOT is $\mathcal{O}(\frac{1}{N})$.

\section{An acceleration version of the IBPUOT}
\label{sec:4}
In this section, we will develop an accelerated version of IBPUOT, denoted as AIBPUOT. In fact, research on accelerating the proximal point algorithm has a long history. Currently, most acceleration strategies for first-order optimization algorithms originate from Nesterov's momentum mechanism \cite{nesterov1988approach}, which constructs an estimate sequence to generate a clever momentum and then uses this momentum to generate an intermediate point for acceleration of iteration. The Bregman extension version of this acceleration method has also been studied in recent years \cite{yang2022bregman,hanzely2021accelerated}.

Inspired by their work, we developed an accelerated version of IBPUOT. In section \ref{sec:4.1}, we introduce the construction of AIBPUOT and provide the algorithmic form. In section \ref{sec:4.2}, we prove the convergence of AIBPUOT and explain the method of setting parameters to achieve acceleration effects.

\subsection{AIBPUOT}
\label{sec:4.1}
In order to develop AIBPUOT, we first need to construct an estimated sequence $\{\phi_{k}(\mathbf{P})\}_{k=0}^{\infty}$, the specific form of which is as follows,
\begin{equation}
	\begin{aligned}
		\phi_{0}(\mathbf{P})&=f(\mathbf{P}^{0})+\sigma D_{h}(\mathbf{P},\mathbf{P}^{0}),\\
		\phi_{k+1}(\mathbf{P})&=(1-\theta_{k})\phi_{k}(\mathbf{P})+\theta_{k}\left(f(\mathbf{P}^{k+1})+\beta_{k}\langle\nabla \mathit{h}(\mathbf{Y}^{k})- \nabla \mathit{h}(\mathbf{P}^{k+1}),\mathbf{P}-\mathbf{P}^{k+1} \rangle -\nu_{k} \right).
	\end{aligned}
	\label{eq:estimate4}
\end{equation}
Where $\sigma$ and $\beta_{k}$ are positive real numbers, $\theta_{k}$ is a real number in $[0,1]$, $\nu_{k}$ is a non-negative real number. According to the estimated sequence of functions $\{\phi_{k}(\mathbf{P})\}_{k=0}^{\infty} $, we construct the Algorithm \ref{alg:AIBPUOT} AIBPUOT.

Firstly, we calculate $\mathbf{Z}^{k}$, which is the minimum point of function $\phi_{k}(\mathbf{P})$ in the estimated sequence, i.e., $\mathbf{Z}^{k}=\arg \min_{\mathbf{P}}\{\phi_{k}(\mathbf{P}) \}$. And then we need to construct an intermediate point $\mathbf{Y}^{k}$ by the following equation
\begin{equation*}
	\mathbf{Y}^{k}=\theta_{k}\mathbf{Z}^{k}+(1-\theta_{k})\mathbf{P}^{k},
\end{equation*}
where $\theta_{k}$ is specified by \eqref{eq:choosetheta}, then we use the $\mathbf{Y}^{k}$ to generate $\mathbf{P}^{k+1}$,
\begin{equation}
	\mathbf{P}^{k+1} = \arg \min_{\mathbf{P}\geqslant 0}
	\langle\bf {C}, \bf{P}\rangle+
	\lambda_{1} \mathrm{KL}(\bf{P}\bf{1}_{m}\mid \bf{a})+
	\lambda_{2} \mathrm{KL}(\bf{P}^{T}\bf{1}_{n} \mid \bf{b})+
	\beta_{k}\mathit{D}_{\mathit{h}}({\mathbf{P},\mathbf{Y}^{k}}).
	\label{eq:abreguot}
\end{equation}
Similar to \eqref{breguot} and \eqref{eq:beuot1}, the above equation can be rewritten in the following form
\begin{equation}
		\mathbf{P}^{k+1}=\arg \min_{\mathbf{P}\geqslant 0}\langle\mathbf{C}^{k}, \bf{P}\rangle+
	\lambda_{1} \mathrm{KL}(\bf{P}\bf{1}_{m}\mid \bf{a})+
	\lambda_{2} \mathrm{KL}(\bf{P}^{T}\bf{1}_{n} \mid \bf{b})+
	\beta_{k} \mathit{h}(\bf{P}),
	\label{eq:abeuot}
\end{equation}
where $\mathbf{C}^{k}=\mathbf{C} -\beta_{k}\log\mathbf{Y}^{k}$. This is an entropy regularized UOT, it can also be solved by \eqref{eq:scalg}. Here, we still adopt an inexact solution of \eqref{eq:abeuot}, as long as $\mathbf{P}^{k+1}$ can meet the following conditions.
\begin{equation}
		0 \in \partial_{\nu_{k}}f(\bf{P}^{k+1})+\beta_{k}(\nabla \mathit{h}(\mathbf{P}^{k+1})- \nabla \mathit{h}(\mathbf{Y}^{k})).
		\label{eq:abuot estinate}
\end{equation}
And then we compute new $\mathbf{Z}^{k+1}$ and iterate until the algorithm converges. The complete algorithm is presented as Algorithm \ref{alg:AIBPUOT}, and for simplicity, we can take $\beta_{k}$ as a constant, and let $\beta_{k}=\beta$. The specific selection method for parameters $\theta_{k}$ and $\mathbf{Z}^{k+1}$ in Algorithm \ref{alg:AIBPUOT} will be discussed in the subsequent convergence proof.

\begin{algorithm}[!h]
	\caption{AIBPUOT($\mathbf{a},\mathbf{b},\lambda_{1},\lambda_{2},\mathbf{C}$)}
	\label{alg:AIBPUOT}
	\renewcommand{\algorithmicrequire}{\textbf{Input:}}
	\renewcommand{\algorithmicensure}{\textbf{Output:}}
	\begin{algorithmic}
		\REQUIRE positive vectors $\mathbf{a}$, $\mathbf{b}$, unbalancedness parameters $\lambda_{1}$,$\lambda_{2}$ and cost matrix $\mathbf{C}$
		\ENSURE $\mathbf{P}^{k}$
		\STATE Let $\bf{v}=\bf{1}_{m}$ and $\mathbf{K}=e^{-\mathbf{C}/\beta}$
		\STATE $\mathbf{P}^{0}=\mathbf{Z}^{0}=\mathbf{1}_{n}\mathbf{1}_{m}^{T}$
		\FOR{k=0,1,3,\dots,N}
		\STATE choose $\theta_{k}\in[0,1)$ satisfying \eqref{eq:choosetheta}, set $\mathbf{Y}^{k}=\theta_{k}\mathbf{Z}^{k}+(1-\theta_{k})\mathbf{P}^{k}$
		\STATE $\mathbf{G}=\mathbf{K}\odot \mathbf{Y}^{k}$
		\WHILE{termination criterion \eqref{eq:abuot estinate} is not met}
		\STATE $\bf{u}=(\frac{\bf{a}}{\bf{K}\bf{v}})^{\frac{\lambda_{1}}{\lambda_{1}+\beta}}$
		\STATE $\bf{v}=(\frac{\bf{b}}{\bf{K}^{T} \bf{u}})^{\frac{\lambda_{2}}{\lambda_{2}+\beta}}$
		\ENDWHILE
		\STATE $\mathbf{P}^{k+1}=Diag(\bf{u})\mathbf{G} \mathit{Diag}(\bf{v})$
		\STATE Set $\phi_{k+1}(\mathbf{P})$ by \eqref{eq:estimate4}, compute $\mathbf{Z}^{k+1}=\arg \min_{P}\{\phi_{k+1}(\mathbf{P})\}$ by \eqref{eq:choose z}
		\ENDFOR
		\RETURN	$\mathbf{P}^{k+1}$
	\end{algorithmic}		
\end{algorithm}

\subsection{Convergence analysis of AIBPUOT}
\label{sec:4.2}
In this subsection, we will study the convergence property of AIBPUOT, as we generate an intermediate point according to Nesterov's idea to serve as the proximal point in subproblem \ref{eq:abreguot} and ultimately converge to the optimal solution of the original problem \eqref{eq:uot}. Therefore, our convergence proof will also differ from that of IBPUOT. Like many other acceleration methods, our proof process is completed using Nesterov's estimate sequence technique.

Before proving, we first introduce the triangle scaling property (TSP) of Bregman distance.
\begin{defn}[triangle scaling property\cite{hanzely2021accelerated}]
	Let $h$ be a convex function that is differentiable on rint dom $h$. The 
	Bregman distance $D_{h}$ has the triangle scaling property if there are constant $\gamma > 0 $ and $\theta\in[0,1]$ such that for all $\mathbf{x},\mathbf{y},\mathbf{z} \in$ rint dom $h$,
	\begin{equation}
		D_{h}((1-\theta)\mathbf{x}+\theta\mathbf{y},(1-\theta)\mathbf{x}+\theta\mathbf{z})\leqslant \tau\theta^{\gamma}D_{h}(\mathbf{y},\mathbf{z}).
		\label{eq:tsp}
	\end{equation}
	Here we call $\gamma$ triangle scaling exponent (TSE) of $D_{h}$, and $\tau$ is called the triangle scaling constants (TSC) of $D_{h}$.
\end{defn}

When $h$ is the entropy kernel function,.i.e,$h(\bf{x})=\sum_{i}\bf{x}_{i}(\log \bf{x}_{i}-1)$, then for any $\theta \in [0,1]$
\begin{equation*}
		D_{h}((1-\theta)\mathbf{x}+\theta\mathbf{y},(1-\theta)\mathbf{x}+\theta\mathbf{z})\leqslant \theta D_{h}(\mathbf{y},\mathbf{z}).
\end{equation*}
In this case, $\tau$=$\gamma$=1.

Now let's start with the convergence proof, first we estimate the difference between $\phi_{k}(\bf{P})$ and $f(\bf{P})$, giving the following lemma.

\begin{lem}
	For all $k\geqslant0$, the  estimate sequence of functions $\{\phi_{k}(\mathbf{P})\}_{k=0}^{\infty}$ be generated by \eqref{eq:estimate4}, we have
	\begin{equation}
		\phi_{k+1}(\mathbf{P})-f(\mathbf{P})\leqslant(1-\theta_{k})(\phi_{k}(\mathbf{P})-f(\mathbf{P})) \quad \forall \mathbf{P} \in dom f.
	\end{equation}
	\label{lem:phi-f}
\end{lem}
\begin{proof}
	According to \eqref{eq:abuot estinate}, we know that there exists a $\mathbf{d}^{k+1}\in\partial_{\nu_{k}}f(\mathbf{P}^{k+1})$ that satisfies $0=\mathbf{d}^{k+1}+\beta_{k}(\nabla \mathit{h}(\mathbf{P}^{k+1})- \nabla \mathit{h}(\mathbf{Y}^{k}))$, therefore, for $ \forall \mathbf{P} \in dom f$, we have
	\begin{equation}
		\begin{aligned}
			f(\mathbf{P}) &\geqslant f(\mathbf{P}^{k+1})+\langle\mathbf{d}^{k+1},\mathbf{P}-\mathbf{P}^{k+1} \rangle -\nu_{k} \\
			&=f(\mathbf{P}^{k+1})+\langle -\beta_{k}(\nabla \mathit{h}(\mathbf{P}^{k+1})- \nabla \mathit{h}(\mathbf{P}^{k})),\mathbf{P}-\mathbf{P}^{k+1} \rangle -\nu_{k}\\
			&=f(\mathbf{P}^{k+1})+\beta_{k}\langle \nabla\mathit{h}(\mathbf{Y}^{k})- \nabla \mathit{h}(\mathbf{P}^{k+1}),\mathbf{P}-\mathbf{P}^{k+1} \rangle -\nu_{k}.
		\end{aligned}
		\label{eq:abptu}
	\end{equation}
	Combing the construction of $\phi_{k}(\mathbf{P})$ in the \eqref{eq:estimate4} and above inequality, we can obtain that
	\begin{equation}
		\begin{aligned}
			\phi_{k+1}(\mathbf{P})-f(\mathbf{P})&=(1-\theta_{k})\phi_{k}(\mathbf{P})\\
			&+\theta_{k}\left(f(\mathbf{P}^{k+1})+\beta_{k}\langle\nabla \mathit{h}(\mathbf{Y}^{k})- \nabla \mathit{h}(\mathbf{P}^{k+1}),\mathbf{P}-\mathbf{P}^{k+1} \rangle -\nu_{k}\right) -f(\mathbf{P})\\
			&=(1-\theta_{k})\left( \phi_{k}(\mathbf{P})-f(\mathbf{P})\right)\\
			&+\theta_{k}\left(f(\mathbf{P}^{k+1})+\beta_{k}\langle\nabla \mathit{h}(\mathbf{Y}^{k})- \nabla \mathit{h}(\mathbf{P}^{k+1}),\mathbf{P}-\mathbf{P}^{k+1} \rangle -\nu_{k}-f(\mathbf{P})\right)\\
			&\leqslant(1-\theta_{k})\left( \phi_{k}(\mathbf{P})-f(\mathbf{P})\right).
		\end{aligned}
	\end{equation}
\end{proof}

According to the above Lemma \ref{lem:phi-f}, it can be deduced that the difference
 $ \phi_{k}(\mathbf{P})-f(\mathbf{P})$ is reduced by a factor $(1-\theta_{k})$ at the $k$-th iteration. We can induce to obtain that
 \begin{equation}
 	\phi_{k}(\mathbf{P})-f(\mathbf{P})\leqslant\rho_k(\phi_{0}(\mathbf{P})-f(\mathbf{P})) \quad \forall \mathbf{P} \in dom f,
 	\label{eq:4.9}
 \end{equation}
 where $\rho_{0}=1$, $\rho_k:=\prod_{i=0}^{k-1}(1-\theta_{i})$ for $k\geqslant1$.

In order to further evaluate the decline of $f(\mathbf{P}^{k})-f(\mathbf{P})$, we also need to study the relationship between $f(\mathbf{P}^{k}) $ and $\phi_{k}^{*}=\phi_{k}(\mathbf{Z}^{k})=\min_{P}\phi_{k}(\mathbf{P})$. Therefore, we will now proceed with the analysis of the estimated sequences.

 In fact, the estimation sequence of functions $\{\phi_{k}(\mathbf{P})\}_{k=0}^{\infty}$ constructed by \eqref{eq:estimate4} can be rewritten in the following form,
 \begin{equation}
 	\phi_{k}(\mathbf{P})=H_{k}(\mathbf{P})+\sigma\rho_{k}D_{h}(\mathbf{P},\mathbf{P}^{0}),
 	\label{eq:affin}
 \end{equation}
 where $H_{k}(\cdot)$ is an affine function. 
 Now we present a lemma, through which we can obtain the relationship between $\phi_{k}(\mathbf{Z}^{k+1})$ and $\phi_{k}(\mathbf{Z}^{k})$.
 
 \begin{lem}
 For the estimation sequence of functions $\{\phi_{k}(\mathbf{P})\}_{k=0}^{\infty}$ constructed by \eqref{eq:estimate4}, and  for any $\mathbf{P} \in$ dom $h$, we have
 	\begin{equation}
	\phi_{k}(\mathbf{P})=\phi_{k}(\mathbf{Z}^{k})+\sigma\rho_{k}D_{h}(\mathbf{P},\mathbf{Z}^{k}), 
 	\end{equation}
 	where $\mathbf{Z}^{k}=\arg\min_{\mathbf{P}}\phi_{k}(\mathbf{P})$.
 	\label{lem:affin}
 \end{lem}
 \begin{proof}
 	From \eqref{eq:affin}, we know $	\phi_{k}(\mathbf{P})=H_{k}(\mathbf{P})+\sigma\rho_{k}D_{h}(\mathbf{P},\mathbf{P}^{0})$, and $H_{k}(\mathbf{P})$ is an affine function.  Since $\mathbf{Z}^{k}$ is the optimal solution of $ \phi_{k}(\mathbf{P})$, so we have
 	\begin{equation*}
 		\nabla H_{k} (\mathbf{Z}^{k})+\sigma\rho_k(\nabla h (\mathbf{Z}^{k})-\nabla h(\mathbf{P}^{0})=0.
 	\end{equation*}
 	Since $H_{k}(\cdot)$ is affine function, so for any $\mathbf{P}\in$ dom $f$ $\cap$ dom $h$, we have
 	\begin{equation*}
 		\begin{aligned}
 				H_{k}(\mathbf{P})&=H_{k}(\mathbf{Z}^{k})+\langle \nabla H_{k}(\mathbf{Z}^{k}),\mathbf{P}-\mathbf{Z}^{k}\rangle\\
 				&=H_{k}(\mathbf{Z}^{k})-\sigma\rho_{k}\langle(\nabla h (\mathbf{Z}^{k})-\nabla h(\mathbf{P}^{0})),\mathbf{P}-\mathbf{Z}^{k}\rangle\\
 				&=H_{k}(\mathbf{Z}^{k})-\sigma\rho_{k}(D_{h}(\mathbf{P},\mathbf{P}^{0})-D_{h}(\mathbf{P},\mathbf{Z}^{k})-D_{h}(\mathbf{Z}^{k},\mathbf{P}^{0})).
 		\end{aligned}
 	\end{equation*}
 	Rewriting the above equation can complete the proof.
 	
 \end{proof}

According to Lemma \ref{lem:affin}, we have
\begin{equation}
	\phi_{k}(\mathbf{Z}^{k+1})=\phi_{k}(\mathbf{Z}^{k})+\sigma\rho_{k}D_{h}(\mathbf{Z}^{k+1},\mathbf{Z}^{k}).
	\label{eq:lem4.2}
\end{equation}

Now we can begin to discuss the relationship between $f(\mathbf{P}^{k}) $ and $\phi_{k}^{*}$. Indeed, we have the following result.

\begin{lem}
	Let $\{\mathbf{P}^{k}\}$ be the sequence generated by the AIBPUOT in Algorithm \ref{alg:AIBPUOT}, $D_{h}$ has the TSE $\gamma\geqslant1$ and TSC $\tau>0$, and $\theta_{k}$ is chosen such that
	\begin{equation}
		\tau\beta_{k}\theta_{k}^{\gamma}=\sigma\rho_{k}(1-\theta_{k}).
		\label{eq:choosetheta}
	\end{equation}
	If $f(\mathbf{P}^{k})\leqslant \phi_{k}^{*} +\delta_{k}$ for some $k\geqslant0$ and $\delta_{k}\geqslant0$, then we have
	\begin{equation}
		f(\mathbf{P}^{k+1})\leqslant \phi_{k+1}^{*} +\delta_{k+1},
		\label{eq:lem4.3.3}
	\end{equation}
	where $\delta_{0}=0$, $\delta_{k+1}:=(1-\theta_{k})\delta_{k}+\nu_{k}$.
	\label{lem:4.3}
\end{lem}
\begin{proof}
	Firstly, we have
	\begin{equation}
		\begin{aligned}
			&(1-\theta_{k})f(\mathbf{P}^{k})+\theta_{k}\left(f(\mathbf{P}^{k+1})+\beta_{k}\langle\nabla \mathit{h}(\mathbf{Y}^{k})- \nabla \mathit{h}(\mathbf{P}^{k+1}),\mathbf{P}-\mathbf{P}^{k+1} \rangle -\nu_{k}\right)\\
			&\geqslant(1-\theta_{k})\left(f(\mathbf{P}^{k+1})+\beta_{k}\langle\nabla \mathit{h}(\mathbf{Y}^{k})- \nabla \mathit{h}(\mathbf{P}^{k+1}),\mathbf{P}^{k}-\mathbf{P}^{k+1} \rangle -\nu_{k}\right)\\
			&+\theta_{k}\left(f(\mathbf{P}^{k+1})+\beta_{k}\langle\nabla \mathit{h}(\mathbf{Y}^{k})- \nabla \mathit{h}(\mathbf{P}^{k+1}),\mathbf{P}-\mathbf{P}^{k+1} \rangle -\nu_{k}\right)\\
			&=f(\mathbf{P}^{k+1})+\beta_{k}\langle\nabla \mathit{h}(\mathbf{Y}^{k})- \nabla \mathit{h}(\mathbf{P}^{k+1}),[(1-\theta_{k})\mathbf{P}^{k}+\theta_{k}\mathbf{Z}^{k+1}]-\mathbf{P}^{k+1} \rangle -\nu_{k}\\
			&=f(\mathbf{P}^{k+1})+\beta_{k}D_{h}((1-\theta_{k})\mathbf{P}^{k}+\theta_{k}\mathbf{Z}^{k+1},\mathbf{P}^{k+1}) + \beta_{k}D_{h}(\mathbf{P}^{k+1},\mathbf{Y}^{k})\\
			&-\beta_{k}D_{h}((1-\theta_{k})\mathbf{P}^{k}+\theta_{k}\mathbf{Z}^{k+1},\mathbf{Y}^{k}) -\nu_{k}\\
			&\geqslant f(\mathbf{P}^{k+1})-\beta_{k}D_{h}((1-\theta_{k})\mathbf{P}^{k}+\theta_{k}\mathbf{Z}^{k+1},\mathbf{Y}^{k}) -\nu_{k}\\
			&=f(\mathbf{P}^{k+1})-\beta_{k}D_{h}((1-\theta_{k})\mathbf{P}^{k}+\theta_{k}\mathbf{Z}^{k+1},(1-\theta_{k})\mathbf{P}^{k}+\theta_{k}\mathbf{Z}^{k}) -\nu_{k}\\
			&\geqslant f(\mathbf{P}^{k+1}) - \tau\beta_{k}\theta_{k}^{\gamma}D_{h}(\mathbf{Z}^{k+1},\mathbf{Z}^{k})-\nu_{k},
		\end{aligned}
		\label{eq:lem4.3.1}
	\end{equation}
	where the first inequality follows from \eqref{eq:abptu} with $\bf{P}=\bf{P}^{k}$, the second equality follows from the three points identity \eqref{three point}, and the last inequality follows from the triangle scaling property of $D_{h}$.
	
	When $k=0$, according to \eqref{eq:estimate4}, we have $f(\mathbf{P}^{0})\leqslant\phi_{k}^{*} $. assume exist some $k\geqslant0$ and $\delta_{k}\geqslant0$, $f(\mathbf{P}^{k})\leqslant \phi_{k}^{*} +\delta_{k}$, then combine it and \eqref{eq:lem4.2}, we obtain that
	\begin{equation}
		\phi_{k}(\mathbf{Z}^{k+1})\geqslant f(\mathbf{P}^{k})+\sigma\rho_{k}D_{h}(\mathbf{Z}^{k+1},\mathbf{Z}^{k})-\delta_{k}.
		\label{eq:lem4.3.2}
	\end{equation}
	Therefore, we can see that
	\begin{equation}
		\begin{aligned}
			\phi_{k+1}^{*}&=\phi_{k+1}(\mathbf{Z}^{k+1})\\
			&=(1-\theta_{k})\phi_{k}(\mathbf{Z}^{k+1})+\theta_{k}\left(f(\mathbf{P}^{k+1})+\beta_{k}\langle\nabla \mathit{h}(\mathbf{Y}^{k})- \nabla \mathit{h}(\mathbf{P}^{k+1}),\mathbf{Z}^{k+1}-\mathbf{P}^{k+1} \rangle -\nu_{k} \right)\\
			&\geqslant (1-\theta_{k})f(\mathbf{P}^{k})+\sigma\rho_{k}(1-\theta_{k})D_{h}D_{h}(\mathbf{Z}^{k+1},\mathbf{Z}^{k})-(1-\theta_{k})\delta_{k}\\
			&+\theta_{k}\left(f(\mathbf{P}^{k+1})+\beta_{k}\langle\nabla \mathit{h}(\mathbf{Y}^{k})- \nabla \mathit{h}(\mathbf{P}^{k+1}),\mathbf{Z}^{k+1}-\mathbf{P}^{k+1} \rangle -\nu_{k} \right)\\
			&\geqslant f(\mathbf{P}^{k+1}) +(\sigma\rho_{k}(1-\theta_{k})-\tau\beta_{k}\theta_{k}^{\gamma}) D_{h}(\mathbf{Z}^{k+1},\mathbf{Z}^{k})-\nu_{k}-(1-\theta_{k})\delta_{k}.
		\end{aligned}
	\end{equation}
	The first inequality follows from \eqref{eq:lem4.3.2}, and the second inequality follows from \eqref{eq:lem4.3.1}. When $\tau\beta_{k}\theta_{k}^{\gamma}=\sigma\rho_{k}(1-\theta_{k})$, \eqref{eq:lem4.3.3} is established. This completes the proof.
\end{proof}

Then, we obtain the theorem about the reduction of the objective value.

\begin{thm}
	Let $\{\mathbf{P}^{k}\}$ be the sequence generated by the AIBPUOT in Algorithm \ref{alg:AIBPUOT}, for any $\mathbf{P}^{*}$, which is the optimal solution of problem \eqref{eq:uot}, i.e.,$f(\mathbf{P})$, we have
	\begin{equation}
		f(\mathbf{P}^{N})-f(\mathbf{P}^{*})\leqslant \rho_{N}\left(f(\mathbf{P}^{0})-f(\mathbf{P}^{*})+\sigma D_{h}(\mathbf{P}^{*},\mathbf{P}^{0}) \right) +\delta_{N}.
	\end{equation}
	\label{thm:4.4}
\end{thm}
\begin{proof}
	From the Lemma \ref{lem:4.3}, we have 
	$f(\mathbf{P}^{N})\leqslant \phi_{N}(\mathbf{Z}^{N}) +\delta_{N},\forall N\geqslant0$. Then we can see that
	\begin{equation}
		\begin{aligned}
			f(\mathbf{P}^{N})-f(\mathbf{P}^{*})&\leqslant\phi_{N}(\mathbf{Z}^{N}) -f(\mathbf{P}^{*})+\delta_{N}\\
			&\leqslant\phi_{N}(\mathbf{P}^{*}) -f(\mathbf{P}^{*})+\delta_{N}\\
			&\leqslant \rho_{N}\left(\phi_{0}(\mathbf{P}^{*}) -f(\mathbf{P}^{*})\right)+\delta_{N}\\
			&\leqslant \rho_{N}(f(\mathbf{P}^{0})-f(\mathbf{P}^{*})+\sigma D_{h}(\mathbf{P},\mathbf{P}^{0}))+\delta_{N}.
		\end{aligned}
	\end{equation}
	The third inequality follows from \eqref{eq:4.9}.
\end{proof}

If we choose $\theta_{k}$ follows from \eqref{eq:choosetheta}, we can see that $0<\theta_{k}<1$, hence $\rho_{N}\rightarrow0$. As long as we can guarantee $\delta_{N}\rightarrow0$ when we choose $\nu_{k}$, then we obtain that $f(\mathbf{P}^{N})$ converges to $f(\mathbf{P}^{*})=\min_{\mathbf{P}}f(\mathbf{P})$ from Theorem \ref{thm:4.4}. If we want to know converge rate of the AIBPUOT, we need to evaluate magnitude of $\rho_{N}$ and $\delta_{N}$. And this is what we are going to do next. Our proof mainly refers to \cite{yang2022bregman}.

\begin{lem}
	 For any $N \geqslant 1 $, we have
	 \begin{equation}
	 	\left(1+\left(\sigma / \tau\right)^{\frac{1}{\gamma}} \sum_{k=0}^{N-1} \beta_{k}^{-\frac{1}{\gamma}}\right)^{-\gamma} \leqslant \rho_{N} \leqslant\left(1+\gamma^{-1}\left(\sigma / \tau\right)^{\frac{1}{\gamma}} \sum_{k=0}^{N-1} \beta_{k}^{-\frac{1}{\gamma}}\right)^{-\gamma}.
	 	\label{eq:lem4.5}
	 \end{equation}
	 In addition, if $\sup_{k}\{ \beta_{k}\} < \infty $ , then $\rho_{N}=\mathcal{O}\left( \left(\sum_{k=0}^{N-1} \beta_{k}^{-\frac{1}{\gamma}}\right)^{-\gamma} \right)$.
	 \label{lem:4.5}
\end{lem}
\begin{proof}
	Due to $\rho_{k+1}=(1-\theta_{k})\rho_{k}$, we can obtain $\theta_{k}=1-\rho_{k+1}/\rho_{k}$. Substituting it in \eqref{eq:choosetheta}, we can see that
	\begin{equation}
		\begin{aligned}
				\tau \beta_{k}\left(1-\rho_{k+1} / \rho_{k}\right)^{\gamma}&=\sigma \rho_{k+1} \\
				 \rho_{k+1}^{-1}-\rho_{k}^{-1}&=\left(\sigma / \tau\right)^{\frac{1}{\gamma}} \beta_{k}^{-\frac{1}{\gamma}} \rho_{k+1}^{\frac{1}{\gamma}-1} .
		\end{aligned}
		\label{eq:lem4.5.1}
	\end{equation}
	And because $\gamma\geqslant1$ and $\rho_{k+1}\leqslant\rho_{k}$, we have
	\begin{equation*}
		\begin{aligned}
				\rho_{k+1}^{-\frac{1}{\gamma}}-\rho_{k}^{-\frac{1}{\gamma}}&= \rho_{k+1}^{1-\frac{1}{\gamma}}\left(\rho_{k+1}^{-1}-\rho_{k+1}^{\frac{1}{\gamma}-1} \rho_{k}^{-\frac{1}{\gamma}}\right) \\
				&\leqslant \rho_{k+1}^{1-\frac{1}{\gamma}}\left(\rho_{k+1}^{-1}-\rho_{k}^{-1}\right) \\
				&= \rho_{k+1}^{1-\frac{1}{\gamma}} \left(\sigma / \tau\right)^{\frac{1}{\gamma}} \beta_{k}^{-\frac{1}{\gamma}} \rho_{k+1}^{\frac{1}{\gamma}-1}\\
				&= \left(\sigma / \tau\right)^{\frac{1}{\gamma}} \beta_{k}^{-\frac{1}{\gamma}}.	
		\end{aligned}
	\end{equation*}
 Summing it from $k = 0 $ to $k = N-1 $, we have
	\begin{equation*}
		\rho_{N}^{-\frac{1}{\gamma}} \leqslant 1+\left(\sigma / \tau\right)^{\frac{1}{\gamma}} \sum_{k=0}^{N-1} \beta_{k}^{-\frac{1}{\gamma}}.
	\end{equation*}
	This is the lower bound of $\rho_{N}$. \\
	And then according to Young’s inequality, we have 
	\begin{equation*}
		\begin{aligned}
				\rho_{k+1}^{\frac{1}{\gamma}-1} \rho_{k}^{-\frac{1}{\gamma}} &\leqslant\left(1-\gamma^{-1}\right) \rho_{k+1}^{-1}+\gamma^{-1} \rho_{k}^{-1} \\
				\rho_{k+1}^{\frac{1}{\gamma}-1} \rho_{k}^{-\frac{1}{\gamma}}&\leqslant \rho_{k+1}^{-1}-\gamma^{-1} \left( \rho_{k+1}^{-1}-\rho_{k}^{-1}\right)\\
				\gamma^{-1} \left( \rho_{k+1}^{-1}-\rho_{k}^{-1}\right)&\leqslant \rho_{k+1}^{-1}-\rho_{k+1}^{\frac{1}{\gamma}-1} \rho_{k}^{-\frac{1}{\gamma}}\\
				\gamma^{-1} \left( \rho_{k+1}^{-1}-\rho_{k}^{-1}\right)&\leqslant \rho_{k+1}^{\frac{1}{\gamma}-1} \left(\rho_{k+1}^{-\frac{1}{\gamma}}-\rho_{k}^{-\frac{1}{\gamma}} \right).
		\end{aligned}
	\end{equation*}
	According to \eqref{eq:lem4.5.1}, we obtain that $\rho_{k+1}^{-\frac{1}{\gamma}}-\rho_{k}^{-\frac{1}{\gamma}} \geqslant \gamma^{-1}\left(\sigma / \tau\right)^{\frac{1}{\gamma}} \beta_{k}^{-\frac{1}{\gamma}}$, Summing it from $k = 0 $ to $k = N-1 $, we have
	\begin{equation*}
		\rho_{N}^{-\frac{1}{\gamma}} \geqslant 1+\gamma^{-1}\left(\sigma / \tau\right)^{\frac{1}{\gamma}} \sum_{k=0}^{N-1} \beta_{k}^{-\frac{1}{\gamma}}.
	\end{equation*}
	This is the upper bound of $\rho_{N}$. And another conclusion can be easily derived from \eqref{eq:lem4.5}. 
\end{proof}

Next we will estimate the magnitude of $\delta_{N}$.

\begin{lem}
	For all $N\geqslant1$, we have
	\begin{equation}
		\delta_{N} \leqslant \left(1+\gamma^{-1}\left(\sigma / \tau\right)^{\frac{1}{\gamma}} \sum_{i=0}^{N-1} \beta_{i}^{-\frac{1}{\gamma}}\right)^{-\gamma} \sum_{k=0}^{N-1}\left(1+\left(\frac{\sigma}{\tau}\right)^{\frac{1}{\gamma}} \sum_{i=0}^{k} \beta_{i}^{-\frac{1}{\gamma}}\right)^{\gamma} \nu_{k}.
		\label{eq:lem4.6}
	\end{equation}
 In addition, if $\{\beta_{k}\}$ is non-increasing and for some $p > 1 $ such that $p\neq \gamma+1$, $\nu_{k} \leqslant \mathcal{O}\left(\frac{\beta_{k}}{(k+1)^{p}}\right)$, 
 then we have $\delta_{N} \leqslant \mathcal{O}\left(\frac{1}{N^{p-1}}\right)$ for all $N\geqslant1$.
 \label{lem:4.6}
\end{lem}
\begin{proof}
	For all $k\geqslant0$, $1-\theta_{k}=\rho_{k+1}/\rho_{k}$, then we have
	\begin{equation*}
		\delta_{k+1}=(\rho_{k+1}/\rho_{k})\delta_{k}+\nu_{k}.
	\end{equation*}
	Dividing this equality by $\rho_{k+1}$, we have $\delta_{k+1}/\rho_{k+1}-\delta_{k}/\rho_{k}=\nu_{k}/\rho_{k+1}$. And then summing it from $k = 0$ to $k=N-1$, we can see that
	\begin{equation*}
		\delta_{N}/\rho_{N}=\sum_{k=0}^{N-1} \nu_{k} / \rho_{k+1}.
	\end{equation*}  
	Using \eqref{eq:lem4.5}, we obtain \eqref{eq:lem4.6}.
	
	In addition, due to $\{\beta_{k}\}$ is non-increasing, we can get that $\sum_{i=0}^{N-1}\beta_{i}^{-\frac{1}{\gamma}}\geqslant N\beta_{0}^{-\frac{1}{\gamma}}$ and $\sum_{i}\beta_{i}^{-\frac{1}{\gamma}}\rightarrow\infty$, so there exists a constant $c>1
	$ satisfy $1+\left(\sigma / \tau\right)^{\frac{1}{\gamma}} \sum_{i=0}^{k} \beta_{i}^{-\frac{1}{\gamma}} \leqslant c\left(\sigma / \tau\right)^{\frac{1}{\gamma}} \sum_{i=0}^{k} \beta_{i}^{-\frac{1}{\gamma}}$. Since $\nu_{k} \leqslant \mathcal{O}\left(\frac{\beta_{k}}{(k+1)^{p}}\right) $, there exists a constant $a$ satisfy $\nu_{k}\leqslant a \beta_{k}/(k+1)^{p}$. 
	We substitute them into \eqref{eq:lem4.6} and obtain that
	\begin{equation*}
		\begin{aligned}
			\delta_{N} 	& \leqslant \left(\gamma^{-1}\left(\sigma / \tau\right)^{\frac{1}{\gamma}} N\beta_{0}^{-\frac{1}{\gamma}}\right)^{-\gamma} \sum_{k=0}^{N-1}\left(c\left(\frac{\sigma}{\tau}\right)^{\frac{1}{\gamma}} \sum_{i=0}^{k} \beta_{i}^{-\frac{1}{\gamma}}\right)^{\gamma} \frac{a\beta_{k}}{(k+1)^{p}} \\
			& \leqslant \frac{\beta_{0} a c^{\gamma} \gamma^{\gamma}}{N^{\gamma}} \sum_{k=0}^{N-1}\left(\sum_{i=0}^{k} \beta_{i}^{-\frac{1}{\gamma}}\right)^{\gamma} \frac{\beta_{k}}{(k+1)^{p}} \\
			& =\frac{\beta_{0} a c^{\gamma} \gamma^{\gamma}}{N^{\gamma}} \sum_{k=0}^{N-1}\left(\sum_{i=0}^{k}\left(\frac{\beta_{k}}{\beta_{i}}\right)^{\frac{1}{\gamma}}\right)^{\gamma} \frac{1}{(k+1)^{p}} \\
			&\leqslant \frac{\beta_{0} a c^{\gamma} \gamma^{\gamma}}{N^{\gamma}} \sum_{k=0}^{N-1}(k+1)^{\gamma-p} .
		\end{aligned}
	\end{equation*}
	And then there exists a constant $b>0$ such that
	\begin{equation*}
		\sum_{k=0}^{N-1}(k+1)^{\gamma-p}=\sum_{k=1}^{N} k^{\gamma-p} \leqslant b \int_{1}^{N} t^{\gamma-p} \mathrm{d}t \leqslant b(\gamma+1-p)^{-1} N^{\gamma+1-p}.
	\end{equation*}
	Combining the above two inequalities, the lemma is proven.
\end{proof}

Combining Lemma \ref{lem:4.5} and Lemma \ref{lem:4.6}, we can give the following
concrete convergence rate of Algorithm \ref{alg:AIBPUOT} AIBPUOT.

\begin{thm}
		Assume that all conditions in Theorem \ref{thm:4.4} and Lemma \ref{lem:4.5} and \ref{lem:4.6} hold. Let $\{\mathbf{P}^{k}\}$ be the sequence generated by the AIBPUOT in Algorithm \ref{alg:AIBPUOT}. For any $\mathbf{P}^{*}$, which is the optimal solution of problem \eqref{eq:uot}, i.e.,$f(\mathbf{P})$, we have
		\begin{equation*}
			f(\mathbf{P}^{N})-f(\mathbf{P}^{*})\leqslant \mathcal{O}\left( \left(\sum_{k=0}^{N-1} \beta_{k}^{-\frac{1}{\gamma}}\right)^{-\gamma} \right) +\mathcal{O}\left(\frac{1}{N^{p-1}}\right).
		\end{equation*}
		In addition, if we choose $\beta_{k}$ as a constant $\beta$ and $p>\gamma+1$, we can get that
		\begin{equation*}
			f(\mathbf{P}^{N})-f(\mathbf{P}^{*})\leqslant\mathcal{O}\left(\frac{1}{N^{\gamma}}\right).
		\end{equation*}
\end{thm}

As we can see, the convergence rate of Algorithm \ref{alg:AIBPUOT} AIBPUOT is determined by the triangle scaling exponent (TSE) of $D_{h}$,i.e., $\gamma$. For AIBPUOT, the $D_{h}$ is based on $h(\bf{x})=\sum_{i}\bf{x}_{i}(\log \bf{x}_{i}-1)$, it implies that $\gamma=1$. In this case, the convergence rate of AIBPUOT is $\mathcal{O}\left(\frac{1}{N}\right)$, which means that no acceleration effect was achieved. But if we give a small $\epsilon>0$, $0<t<1$, for all $\theta\in[\epsilon^{\frac{1}{t}},1]$, we have
\begin{equation}
	D_{h}((1-\theta)\mathbf{x}+\theta\mathbf{y},(1-\theta)\mathbf{x}+\theta\mathbf{z})\leqslant \theta D_{h}(\mathbf{y},\mathbf{z})\leqslant\frac{1}{\theta^{t}}\theta^{1+t} D_{h}(\mathbf{y},\mathbf{z})\leqslant\frac{1}{\epsilon}\theta^{1+t} D_{h}(\mathbf{y},\mathbf{z}),
\end{equation}
which implies that \eqref{eq:tsp} hold for all $\theta\in[\epsilon^{\frac{1}{t}},1]$ with $\gamma=1+t $ and $\tau=\frac{1}{\epsilon}$. So if we want get a effective AIBPUOT, we need to give a small $\epsilon>0$, choose a $\theta_{k}$ satisfies \eqref{eq:choosetheta} and $\theta_{k}\in[\epsilon^{\frac{1}{t}},1]$. In this way, we have $\gamma=1+t$ the algorithm AIBPUOT can achieve acceleration in theory. But, according to our experiment,  if we set $\tau=\frac{1}{\epsilon}$ , we don't get a desired effect. So we finally adopt a heuristic strategy to choose $\tau$. Specifically, we firstly set $\tau =1$ and then substitute it by $\tau=2\tau$ if $\tau\theta_{k}^{t}<0.125$.

\begin{remk}[Explicit calculations of $\mathbf{Z}^{k+1}$]
	According to the definition of $\mathbf{Z}^{k+1}$, which is the optimal solution of $\phi_{k+1}(\mathbf{P})$, so we have 
	\begin{equation*}
		\begin{aligned}
			\mathbf{Z}^{k+1}&=\arg\min_{P}\left\{\phi_{k+1}(\mathbf{P})\right\}\\
			&=\arg\min_{P}\left\{ (1-\theta_{k})\phi_{k}(\mathbf{P})+\theta_{k}\beta_{k}\langle\nabla \mathit{h}(\mathbf{Y}^{k})- \nabla \mathit{h}(\mathbf{P}^{k+1}),\mathbf{P}\rangle \right\}\\
			&=\arg\min_{P}\left\{ (1-\theta_{k})(\phi_{k}(\mathbf{Z}^{k})+\sigma\rho_{k}D_{h}(\mathbf{P},\mathbf{Z}^{k}))+\theta_{k}\beta_{k}\langle\nabla \mathit{h}(\mathbf{Y}^{k})- \nabla \mathit{h}(\mathbf{P}^{k+1}),\mathbf{P}\rangle \right\}\\
			&=\arg\min_{P}\left\{ (1-\theta_{k})\sigma\rho_{k}D_{h}(\mathbf{P},\mathbf{Z}^{k})+\theta_{k}\beta_{k}\langle\nabla \mathit{h}(\mathbf{Y}^{k})- \nabla \mathit{h}(\mathbf{P}^{k+1}),\mathbf{P}\rangle \right\}\\
			&=\arg\min_{P}\left\{ \tau\beta_{k}\theta_{k}^{\gamma}D_{h}(\mathbf{P},\mathbf{Z}^{k})+\theta_{k}\beta_{k}\langle\nabla \mathit{h}(\mathbf{Y}^{k})- \nabla \mathit{h}(\mathbf{P}^{k+1}),\mathbf{P}\rangle \right\}\\
			&=\arg\min_{P}\left\{ \tau\theta_{k}^{\gamma-1}D_{h}(\mathbf{P},\mathbf{Z}^{k})+\langle\nabla \mathit{h}(\mathbf{Y}^{k})- \nabla \mathit{h}(\mathbf{P}^{k+1}),\mathbf{P}\rangle \right\}.
		\end{aligned}
	\end{equation*}
	The second equality follows from \eqref{eq:estimate4}, the third equality follows from Lemma \ref{lem:affin}, the fifth equality follows from \eqref{eq:choosetheta}. According to the optimal condition, we can obtain that
	\begin{equation*}
		\nabla h(\mathbf{Z}^{k+1})=\nabla h(\mathbf{Z}^{k})-\tau^{-1}\theta_{k}^{1-\gamma }(\nabla \mathit{h}(\mathbf{Y}^{k})- \nabla \mathit{h}(\mathbf{P}^{k+1})).
	\end{equation*}
	Therefore $\mathbf{Z}^{k+1}=\nabla h^{-1} \left(\nabla h(\mathbf{Z}^{k})-\tau^{-1}\theta_{k}^{1-\gamma }(\nabla \mathit{h}(\mathbf{Y}^{k})- \nabla \mathit{h}(\mathbf{P}^{k+1})) \right) $. Since $h(\bf{x})=\sum_{i}\bf{x}_{i}(\log \bf{x}_{i}-1)$, we have that $h^{-1}(\bf{x})=\sum_{i}e^{\bf{x}_{i}}$. Then we can see that
	\begin{equation}
		\mathbf{Z}^{k+1}=\mathbf{Z}^{k+1}\odot(\frac{\mathbf{P}^{k+1}}{\mathbf{Y}^{k}})^{\tau^{-1}\theta_{k}^{1-\gamma }}.
		\label{eq:choose z}
	\end{equation}
\end{remk}

\section{Numerical experiments}
\label{sec:5}

In this section, we will demonstrate the advantages of IBPUOT over the Scaling algorithm and MM algorithm \cite{chapel2021unbalanced} in terms of convergence, as well as in sparsity compared to the Scaling algorithm. In addition, we will also show the impact of different parameters $\beta$ on the convergence speed. We will then present the acceleration performance of our AIBPUOT.

\subsection{Implementation details.}
\label{sec:5.1}

 We compare the performance of different algorithms by calculating the UOT distance between two one-dimensional Gaussian distributions. The two Gaussian distributions we have selected are the source distribution $\mathcal{N}(20,5)+\mathcal{N}(50,9)$ and the target distribution $\mathcal{N}(60,10)$. $\mathcal{N}(\mu,\sigma^{2})$ is the probability density function for the one-dimensional Gaussian distribution, $\mu$ and $\sigma^{2}$ represent its mean and variance respectively.  We present these two distributions in Figure \ref{fig1}, where the blue represents the source distribution and the red represents the target distribution.  Input vectors $\mathbf{a}$ and $\mathbf{b}$ are the values of the source distribution and the target distribution discretized on the interval [1,100], with a grid size of 1. 
 
\begin{figure*}[htbp]
	\centering
	\includegraphics[width=0.6\columnwidth,height=0.3\linewidth]{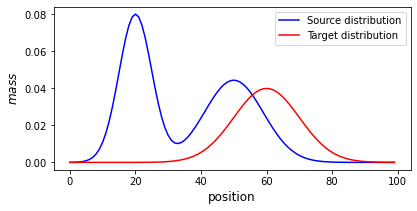}
	\caption{marginal distribution}
	\label{fig1}
\end{figure*}
 
 In addition, we set $\lambda_{1}=\lambda_{2}=1$ in the original problem \eqref{eq:uot}. Since we want to demonstrate the performance of various algorithms in terms of convergence, it is crucial to have the exact solution for the UOT distance between two distributions. However, currently there is no algorithm that computes the exact solution for the UOT problem based on KL divergence. Therefore, we choose to use an approximation that is close enough to the true solution to compare the convergence speeds of different algorithms, we call it 'aproxtruth'. Here, we select the result obtained by IBPUOT with $\beta_{k}=0.005$ after 10,000 outer iterations as this approximate exact solution. 
 In Figure \ref{fig2}, we demonstrate the convergence of IBPUOT in this computational process. The reason for choosing this value as an approximate exact solution is because we found in experiments that our algorithm has a higher accuracy compared to other algorithms for solving the UOT problem based on KL divergence (e.g., MM algorithm, Scaling algorithm). This will also be demonstrated in our subsequent experiments. 
 
 \begin{figure*}[htbp]
 	\centering
 	\includegraphics[width=0.7\columnwidth,height=0.3\linewidth]{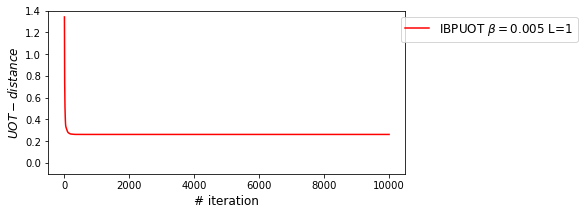}
 	\caption{Calculate the approximate exact solution}
 	\label{fig2}
 \end{figure*}
 
  To better compare the performance between algorithms, we choose to set the number of internal iterations $L=1$ for inexact proximal operator computation in IBPUOT and AIBPUOT in the subsequent algorithm comparisons. In this way, IBPUOT and AIBPUOT only perform one scaling iteration in each outer iteration, without the need to calculate the accuracy of internal iterations for inexact solutions, which can save a lot of computational costs. Experiments show that the algorithm can still converge and perform well in this scenario. 
  
\subsection{Comparison results}
\subsubsection{Convergence rate of the algorithm}
 In this section, we  will compare the algorithm convergence rates of IBPUOT algorithm with $\beta=1$, MM algorithm, Scaling algorithm with $\epsilon=0.01$, and $\epsilon=0.001$. We calculated the UOT distance of the problem given in Section \ref{sec:5.1} using these algorithms, and then calculate the difference between the solution obtained by the algorithm and the approximate accurate solution we selected. The relationship between the difference and the number of iterations is shown in Figure \ref{fig3} in the logarithmic domain. It is emphasized again that the number of iterations here is external iterations, but the number of internal iterations $L=1$ in IBPUOT. 
 
  \begin{figure*}[htbp]
 	\centering
 	\includegraphics[width=0.75\columnwidth,height=0.35\linewidth]{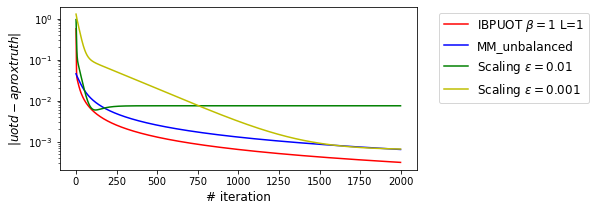}
 	\caption{Comparison of convergence rates of various algorithms }
 	\label{fig3}
 \end{figure*}
 
 Based on Figure \ref{fig3}, we can see that the convergence speed and accuracy performance of our IBPUOT at parameter $\beta=1$ are already better than the MM algorithm for solving similar problems. On the other hand,the Scaling algorithm still does not perform as well as IBPUOT at $\beta=1$ in terms of convergence speed and accuracy when $\epsilon=0.001$. This further confirms the theoretical advantages of our IBPUOT. Even taking a medium-sized parameter $\beta$ in IBPUOT can still yield good results.
 
 In addition, we also compared the influence of different parameters $\beta$ on the convergence speed of IBPUOT. In order to compare the influence of $\beta$, we set the internal iterations of all algorithms to 1, i.e., $L=1$.  We present the results in Figure \ref{fig4}.
   \begin{figure*}[htbp]
 	\centering
 	\includegraphics[width=0.75\columnwidth,height=0.35\linewidth]{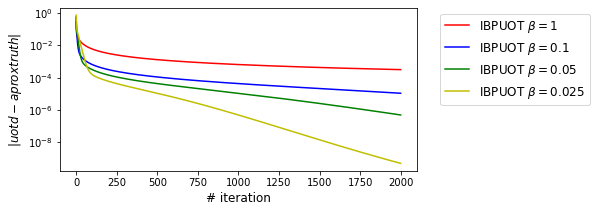}
 	\caption{Comparison of convergence rates of IBPUOT with different $\beta$ }
 	\label{fig4}
 \end{figure*}
 
 Although in general, the smaller the value of $\beta$ selected, the faster the overall convergence to the true solution, we can observe that in the first hundred iterations with $\beta=0.025$, the convergence speed is slower compared to larger values of $\beta$. This is because when we only perform one inner iteration, in the early stages of the algorithm, it leads to a large error in approximating the proximal operator of the small $\beta$, resulting in slower convergence of the algorithm. Therefore, when only performing one inner iteration and a small number of iteration processes, it is not advisable to choose a very small $\beta$.

\subsubsection{Sparsity of the Transport Plan}
 In applications such as color transfer and domain adaptation, we need accurate and sparse transfer plans. Therefore, in this section, we will test the sparsity of the transfer plans of IBPUOT and the Scaling algorithm using the two distributions in Figure 1. We will show the comparison results of the sparsity of the transfer plans in Figure \ref{fig5}. 
\begin{figure}[htbp]
	\centering
	\subfigure[IBPUOT $\beta=0.1$]
	{
		\begin{minipage}[b]{.3\linewidth}
			\centering
			\includegraphics[scale=0.4]{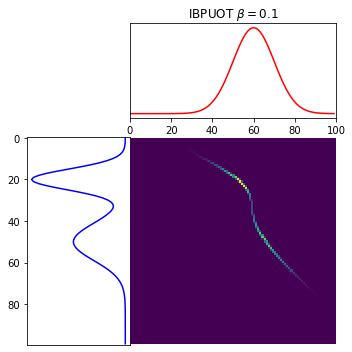}
		\end{minipage}
	}
	\subfigure[Scaling $\epsilon=0.01$]
	{
		\begin{minipage}[b]{.3\linewidth}
			\centering
			\includegraphics[scale=0.4]{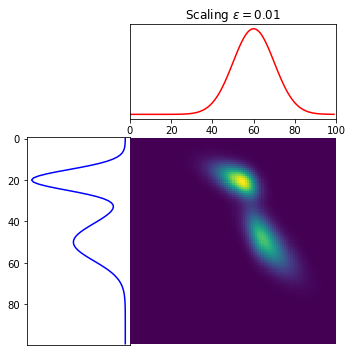}
		\end{minipage}
	}
	\subfigure[Scaling $\epsilon=0.001$]
	{
		\begin{minipage}[b]{.3\linewidth}
			\centering
			\includegraphics[scale=0.4]{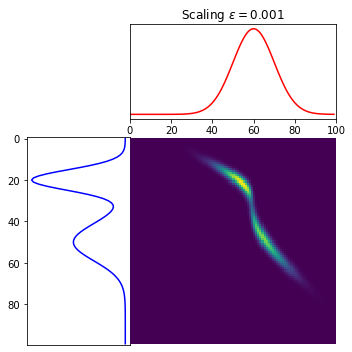}
		\end{minipage}
	}
	\caption{The transport plan of IBPUOT and Scaling algorithm}
	\label{fig5}
\end{figure}

According to Figure \ref{fig5}, we can see that the proposed IBPUOT always converges to a sparse optimal transportation plan. In contrast, for the Scaling algorithm, if we choose a larger regularization parameter $\epsilon$, then the optimal transportation plan is blurred and lacks sparsity. On the other hand, if we choose a smaller regularization parameter $\epsilon$, although sparsity is better, more iterations are needed for the algorithm to converge. Additionally, we can see that the performance in terms of sparsity at $\epsilon=0.001$ is not as good as the IBPUOT with $\beta=0.1$. 
If we need to use Scaling algorithms with smaller regularization parameters to obtain sparser solutions, numerical underflow will occur. For example, in this problem, when the regularization parameter $\epsilon=0.0001$, numerical underflow occurs. However, our IBPUOT can avoid the need for carefully selecting the regularization parameter.

 \subsubsection{Acceleration effect}
 
  In this section, we will compare the convergence speed of IBPUOT and its accelerated version AIBPUOT, in order to demonstrate the effectiveness of our acceleration method. We solve the problem in Section \ref{sec:5.1} using two algorithms and calculate the difference between the corresponding UOT distances and our approximate exact solution. We compare the trend of the difference and the number of iterations in the logarithmic domain. The iteration count here remains the number of outer iterations, but we set the inner iterations for IBPUOT and AIBPUOT to satisfy $L=1$. Here, we compared the acceleration effect of AIBPUOT relative to IBPUOT under different $\beta$.  We show the comparison results in Figure \ref{fig6}
  
  \begin{figure}[htbp]
  		\centering
  	\subfigure[IBPUOT and AIBPUOT with $\beta=1$]
  	{
  		\begin{minipage}[b]{.5\linewidth}
  			\centering
  			\includegraphics[scale=0.5]{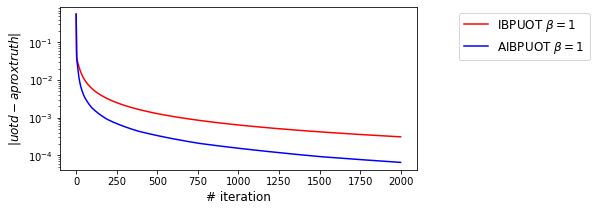}
  		\end{minipage}
  	}
  	\subfigure[IBPUOT and AIBPUOT with $\beta=0.1$]
  	{
  		\begin{minipage}[b]{.5\linewidth}
  			\centering
  			\includegraphics[scale=0.5]{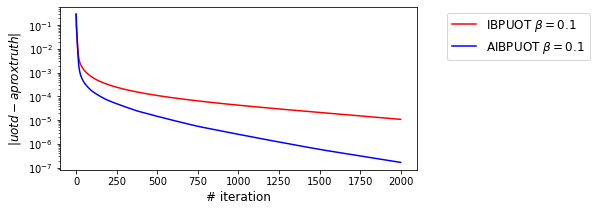}
  		\end{minipage}
  	}
  	\caption{The transport plan of IBPUOT and Scaling algorithm}
  	\label{fig6}
  \end{figure}

We can see that when the same parameter $\beta$ is selected, our AIBPUOT can achieve higher accuracy with the same number of external iterations, and the acceleration effect is very obvious. This is very advantageous for some computationally intensive problems.

\section{Conclusion}

In this paper, we propose an algorithm IBPUOT for solving UOT based on the inexact Bregman proximal point method, and prove its theoretical convergence rate to be $\mathcal{O}(\frac{1}{N})$. Through numerical experiments, our algorithm shows advantages in stability and convergence compared to the Scaling algorithm. Additionally, IBPUOT can converge to the true solution even if only one inner iteration is performed in each round. Furthermore, we develop an accelerated version of IBPUOT called AIBPUOT, and prove its convergence rate to be $\mathcal{O}(\frac{1}{N^\gamma})$. AIBPUOT exhibits acceleration when the triangle scaling exponent $\gamma>1$. 
We also conducted numerical experiments, which showed that AIBPUOT does indeed have an acceleration effect in solving the UOT problem.

\section*{Acknowledgment}
 Jun Liu was supported by the National Natural Science Foundation of China (No. 12371527). Faqiang Wang was supported by the Fundamental Research Funds for the Central Universities. Faqiang Wang was also supported by the National Natural Science Foundation of China (No. 12101058). Li Cui was supported by the National Natural Science Foundation of China (No. 12171043).

\bibliography{reference}

\begin{thebibliography}{10}

\bibitem{villani2021topics}
C{\'e}dric Villani.
\newblock {\em Topics in optimal transportation}, volume~58.
\newblock American Mathematical Soc., 2021.

\bibitem{brenier1991polar}
Yann Brenier.
\newblock Polar factorization and monotone rearrangement of vector-valued
  functions.
\newblock {\em Communications on pure and applied mathematics}, 44(4):375--417,
  1991.

\bibitem{rabin2015convex}
Julien Rabin and Nicolas Papadakis.
\newblock Convex color image segmentation with optimal transport distances.
\newblock In {\em Scale Space and Variational Methods in Computer Vision: 5th
  International Conference, SSVM 2015, L{\`e}ge-Cap Ferret, France, May 31-June
  4, 2015, Proceedings 5}, pages 256--269. Springer, 2015.

\bibitem{rubner2000earth}
Yossi Rubner, Carlo Tomasi, and Leonidas~J Guibas.
\newblock The earth mover's distance as a metric for image retrieval.
\newblock {\em International journal of computer vision}, 40:99--121, 2000.

\bibitem{solomon2014wasserstein}
Justin Solomon, Raif Rustamov, Leonidas Guibas, and Adrian Butscher.
\newblock Wasserstein propagation for semi-supervised learning.
\newblock In {\em International Conference on Machine Learning}, pages
  306--314. PMLR, 2014.

\bibitem{frogner2015learning}
Charlie Frogner, Chiyuan Zhang, Hossein Mobahi, Mauricio Araya, and Tomaso~A
  Poggio.
\newblock Learning with a wasserstein loss.
\newblock {\em Advances in neural information processing systems}, 28, 2015.

\bibitem{schiebinger2019optimal}
Geoffrey Schiebinger, Jian Shu, Marcin Tabaka, Brian Cleary, Vidya Subramanian,
  Aryeh Solomon, Joshua Gould, Siyan Liu, Stacie Lin, Peter Berube, et~al.
\newblock Optimal-transport analysis of single-cell gene expression identifies
  developmental trajectories in reprogramming.
\newblock {\em Cell}, 176(4):928--943, 2019.

\bibitem{lee2019parallel}
John Lee, Nicholas~P Bertrand, and Christopher~J Rozell.
\newblock Parallel unbalanced optimal transport regularization for large scale
  imaging problems.
\newblock {\em arXiv preprint arXiv:1909.00149}, 2019.

\bibitem{li2022application}
Da~Li, Michael~P Lamoureux, and Wenyuan Liao.
\newblock Application of an unbalanced optimal transport distance and a mixed
  l1/wasserstein distance to full waveform inversion.
\newblock {\em Geophysical Journal International}, 230(2):1338--1357, 2022.

\bibitem{yang2018scalable}
Karren~D Yang and Caroline Uhler.
\newblock Scalable unbalanced optimal transport using generative adversarial
  networks.
\newblock {\em arXiv preprint arXiv:1810.11447}, 2018.

\bibitem{janati2020spatio}
Hicham Janati, Marco Cuturi, and Alexandre Gramfort.
\newblock Spatio-temporal alignments: Optimal transport through space and time.
\newblock In {\em International Conference on Artificial Intelligence and
  Statistics}, pages 1695--1704. PMLR, 2020.

\bibitem{cuturi2013sinkhorn}
Marco Cuturi.
\newblock Sinkhorn distances: Lightspeed computation of optimal transport.
\newblock {\em Advances in neural information processing systems}, 26, 2013.

\bibitem{chizat2018scaling}
Lenaic Chizat, Gabriel Peyr{\'e}, Bernhard Schmitzer, and Fran{\c{c}}ois-Xavier
  Vialard.
\newblock Scaling algorithms for unbalanced optimal transport problems.
\newblock {\em Mathematics of Computation}, 87(314):2563--2609, 2018.

\bibitem{bauschke2000dykstras}
Heinz~H Bauschke and Adrian~S Lewis.
\newblock Dykstras algorithm with bregman projections: A convergence proof.
\newblock {\em Optimization}, 48(4):409--427, 2000.

\bibitem{benamou2015iterative}
Jean-David Benamou, Guillaume Carlier, Marco Cuturi, Luca Nenna, and Gabriel
  Peyr{\'e}.
\newblock Iterative bregman projections for regularized transportation
  problems.
\newblock {\em SIAM Journal on Scientific Computing}, 37(2):A1111--A1138, 2015.

\bibitem{altschuler2017near}
Jason Altschuler, Jonathan Niles-Weed, and Philippe Rigollet.
\newblock Near-linear time approximation algorithms for optimal transport via
  sinkhorn iteration.
\newblock {\em Advances in neural information processing systems}, 30, 2017.

\bibitem{pham2020unbalanced}
Khiem Pham, Khang Le, Nhat Ho, Tung Pham, and Hung Bui.
\newblock On unbalanced optimal transport: An analysis of sinkhorn algorithm.
\newblock In {\em International Conference on Machine Learning}, pages
  7673--7682. PMLR, 2020.

\bibitem{xie2020fast}
Yujia Xie, Xiangfeng Wang, Ruijia Wang, and Hongyuan Zha.
\newblock A fast proximal point method for computing exact wasserstein
  distance.
\newblock In {\em Uncertainty in artificial intelligence}, pages 433--453.
  PMLR, 2020.

\bibitem{yang2022bregman}
Lei Yang and Kim-Chuan Toh.
\newblock Bregman proximal point algorithm revisited: A new inexact version and
  its inertial variant.
\newblock {\em SIAM Journal on Optimization}, 32(3):1523--1554, 2022.

\bibitem{nesterov1988approach}
Yurii Nesterov.
\newblock On an approach to the construction of optimal methods of minimization
  of smooth convex functions.
\newblock {\em Ekonomika i Mateaticheskie Metody}, 24(3):509--517, 1988.

\bibitem{nesterov1983method}
Yurii~Evgen'evich Nesterov.
\newblock A method of solving a convex programming problem with convergence
  rate o$\backslash$bigl(k\^{}2$\backslash$bigr).
\newblock In {\em Doklady Akademii Nauk}, volume 269, pages 543--547. Russian
  Academy of Sciences, 1983.

\bibitem{guler1992new}
Osman G{\"u}ler.
\newblock New proximal point algorithms for convex minimization.
\newblock {\em SIAM Journal on Optimization}, 2(4):649--664, 1992.

\bibitem{yan2020bregman}
Shen Yan and Niao He.
\newblock Bregman augmented lagrangian and its acceleration.
\newblock {\em arXiv preprint arXiv:2002.06315}, 2020.

\bibitem{hanzely2021accelerated}
Filip Hanzely, Peter Richtarik, and Lin Xiao.
\newblock Accelerated bregman proximal gradient methods for relatively smooth
  convex optimization.
\newblock {\em Computational Optimization and Applications}, 79:405--440, 2021.

\bibitem{chu2023efficient}
Hong~TM Chu, Ling Liang, Kim-Chuan Toh, and Lei Yang.
\newblock An efficient implementable inexact entropic proximal point algorithm
  for a class of linear programming problems.
\newblock {\em Computational Optimization and Applications}, 85(1):107--146,
  2023.

\bibitem{lemaire1995convergence}
B~Lemaire.
\newblock On the convergence of some iterative methods for convex minimization.
\newblock In {\em Recent Developments in Optimization: Seventh French-German
  Conference on Optimization}, pages 252--268. Springer, 1995.

\bibitem{kiwiel1997proximal}
Krzysztof~C Kiwiel.
\newblock Proximal minimization methods with generalized bregman functions.
\newblock {\em SIAM journal on control and optimization}, 35(4):1142--1168,
  1997.

\bibitem{teboulle1997convergence}
Marc Teboulle.
\newblock Convergence of proximal-like algorithms.
\newblock {\em SIAM Journal on Optimization}, 7(4):1069--1083, 1997.

\bibitem{solodov2001unified}
Mikhail~V Solodov and Benar~Fux Svaiter.
\newblock A unified framework for some inexact proximal point algorithms.
\newblock {\em Numerical functional analysis and optimization},
  22(7-8):1013--1035, 2001.

\bibitem{schmidt2011convergence}
Mark Schmidt, Nicolas Roux, and Francis Bach.
\newblock Convergence rates of inexact proximal-gradient methods for convex
  optimization.
\newblock {\em Advances in neural information processing systems}, 24, 2011.

\bibitem{chapel2021unbalanced}
Laetitia Chapel, R{\'e}mi Flamary, Haoran Wu, C{\'e}dric F{\'e}votte, and
  Gilles Gasso.
\newblock Unbalanced optimal transport through non-negative penalized linear
  regression.
\newblock {\em Advances in Neural Information Processing Systems},
  34:23270--23282, 2021.

\end{thebibliography}

\end{document}